\documentclass[11pt]{amsart}

\usepackage{amsmath}
\usepackage{url}
\usepackage{amsthm}
\usepackage{amsfonts}
\usepackage{hyperref}
\usepackage{accents}
\usepackage{xcolor}

\usepackage{graphics}
\usepackage{wrapfig}
\usepackage{graphicx}

\newtheorem{theorem}{Theorem}[section]
\newtheorem{lemma}[theorem]{Lemma}
\newtheorem{proposition}[theorem]{Proposition}
\newtheorem{corollary}[theorem]{Corollary}

\newtheorem{prop}{Proposition}[section]
\newtheorem{remark}[prop]{Remark}
\newtheorem{definition}[prop]{Definition}

\makeatletter \@addtoreset{equation}{section} \makeatother

\newcommand{\A}{\textbf{A}}

\newcommand{\ddiv}{\text{div}}

\newcommand{\circA}{\accentset{\circ}{A}}
\def\th{\theta}

\def\n{\nabla}
\def\p{\partial}
\def\<{\langle}
\def\>{\rangle}

\begin{document}

\title[]{On the Morse Index with Constraints II: Applications}

\author{Hung Tran}
\address{Department of Mathematics and Statistics,
	Texas Tech University,
	Lubbock, TX 79409}

\author{Detang Zhou}
\address{Instituto de Matem\'atica e Estat\' \i stica, Universidade Federal Fluminense, Rua Professor Marcos Waldemar de Freitas Reis, Bloco H - Campus do Gragoat\'a, S\~ao Domingos, 24.210-201, Niter\'oi, RJ - BRAZIL}
\thanks{The first author was partially supported by a Simons Foundation Collaboration Grant}
\thanks{The second author was partially  supported by Faperj and CNPq of Brazil.}
\date{}

\maketitle
\begin{abstract} This is the second paper in our sequence. Here, we apply our abstract Morse index formulation developed in the previous paper \cite{TZindexconstraintsI20} to study several optimization set-ups with constraints, including type I or/and type II considerations. A common theme is that critical points belong to the family of capillary surfaces, defined by constant mean curvature and intersecting the ambient manifold at a fixed angle. In each case, we classify how the general index is related to the index with a constraint. For capillary surfaces in a Euclidean ball, we obtain an index estimate which recovers stability results of G. Wang and C. Xia \cite{WX19}and J. Gou and C. Xia \cite{GX20} as special cases. By considering a family of examples, we show that inequality is also sharp. Furthermore, we precisely determine indices with constraints for important examples such as the critical catenoid, round cylinders in a ball, and CMC surfaces with constant curvature in a sphere. 
	
\end{abstract}

\section{Introduction}
Let $(\Omega,g)$ be a $m$-dimensional Riemannian manifold and $\Sigma$ be an $n$-dimensional differentiable manifold possibly with boundary $\partial \Sigma$. On the set $ \mathcal{I}(\Sigma, \Omega)$ of all immersions of $\Sigma$ into $M$ there is an important functional $A$ which is defined as the area($n$-volume)  of  the induce metric $i^*(g)$ on $\Sigma$ by $i\in\mathcal{I}(\Sigma, M)$. It is well known that the first variation of $A$ is 
\begin{equation}\label{eqn1}
A'(0)= \int_\Sigma [\textrm{div}(Y^t)-\langle Y, \vec{H}\rangle] d\mu= \int_{\partial\Sigma }\langle Y^t, \eta\rangle ds- \int_\Sigma \langle Y, \vec{H}\rangle d\mu.
\end{equation}
Here $Y$ is the variation vector field, $Y^t$ is the tangent part of $Y$ and $\eta$ is the outward normal vector along $\partial\Sigma$.\\

From the viewpoint of Morse theory, it is essential to study  the second variation at the critical points of the functional $A$ which  involves a symmetric bi-linear form in a suitable function space. The second variation has the following structural formula, for smooth functions $u, v$, 
\begin{equation}
\label{generalindexform}
Q(u,v) = \int_{\Sigma}\Big(\left\langle{\nabla u,\nabla v}\right\rangle-p uv\Big) d\mu-\int_{\partial \Sigma}q uv ds. 
\end{equation} 
Here functions $p, q$, determined by the geometry of $\Sigma$, presumably depend on the particular variational problem we consider. Let $\text{MI}(Q)$ denote the Morse index of the bi-linear form $Q$ on the vector space of smooth functions, $C^\infty(\Sigma)$. It is the maximal dimension of a subspace on which $Q(\cdot, \cdot)$ is negative definite. Generally, $\text{MI}(Q)$ is the index \textit{without any constraint} and essentially counts the number of distinct deformations that decrease the area to the second order. \\

For manifolds with boundaries i.e. $\partial \Sigma\ne \emptyset$,  from (\ref{eqn1}) one could consider  suitable conditions of  partitioning of a convex body by hyper-surfaces with least area, for example, under a type I to type II constraint \cite{BS79}. Type I requires the partitions to have prescribed volume while type II preserves the wetting boundary area. The constraint of preserving enclosed volume is particularly popular in literature, normally associated with constant-mean-curvature (CMC) hypersurfaces. \\

The index with a constraint is the index of $Q(\cdot, \cdot)$ restricted to a smaller function space. The relation between these notions, with and without a constraint, has only been studied at special cases \cite{BB00, LR89, koise01, vogel87, souam19}. We will use the abstract formulation in \cite{TZindexconstraintsI20} to  determine such relation for several variational problems in geometry for compact hypersurfaces. The non-compact case will be address elsewhere. \\

Towards that goal, let $\Sigma^n \subset \Omega^{n+1}$ are manifolds possibly with boundaries. Furthermore, let's assume that $\Sigma$ is a two-sided critical point of a geometric functional with some constraint. In the presence of boundaries, we assume $\Omega$ is diffeomorphic to a Euclidean ball and the immersion is proper; that is, $\partial\Sigma=\Sigma\cap \partial \Omega$. Also all critical points considered belong to the family of capillary surfaces, defined as having CMC and intersecting the ambient manifold at a constant angle. Thus, when there is no boundaries, the latter is vacuously true and closed CMC hypersurfaces belong to it. \\

For the history of capillary surfaces, we refer to an article of  Finn-McCuan-Wente \cite{FMW} and Finn's book \cite{Finn2} for a survey about the mathematical theory of capillary surfaces.  This subject has received plenty of interests recently; see the work of G. Wang and C. Xia \cite{HC18}, H. Li and C. Xiong \cite{WX19} where they proved that all weakly stable capillary hypersurfaces in the unit Euclidean ball must be totally umbilical. When the angle is right, it is just a free boundary hypersurface,  a subject of great current interest,  which has produced many beautiful results;  see, for example, the works of A. Fraser-M.Li\cite{FL14} and A.Fraser-R. Schoen\cite 
{FS11} and \cite{FS16}, M. Li and X. Zhou \cite{LZ16regularity}, D. Maximo, I. Nunes, and G. Smith \cite{MNS16}, and the references therein.\\

Also, it is interesting that for all setups that we will consider the index form stays the same. That is, regarding equation (\ref{generalindexform}), 
\begin{align*}
p &:= \text{Rc}^{\Omega}(\nu,\nu)+|\A^\Sigma|^2,\\
q &:= \frac{1}{\sin\th}\A^{\partial\Omega}(\bar{\nu},\bar{\nu})+\cot\th \A^{\Sigma}(\eta,\eta).
\end{align*}

Here, $\text{Rc}$ denotes the Ricci curvature of $\Omega$; $\nu$ is a unit normal vector of $\Sigma\subset \Omega$; $\A$ is the second fundamental form; $\th\in (0,\pi)$ is the constant intersecting angle; $\bar{\nu}$ is a unit normal vector of $\partial\Sigma\subset \partial\Omega$; $\eta$ is a unit normal vector of $\partial\Sigma\subset\Sigma$.\\

First, we consider the type I partitioning of a convex body. The functional is a linear combination of the area of the hypersurface and the wetting area defined by its boundary components inside $\partial\Omega$. The constraint is fixing the enclosed volume of the region bounded by $\Sigma$ and parts of $\partial \Omega$. It can be shown that a critical point is characterised by having CMC and intersecting $\partial \Omega$ at a constant angle, so it is a capillary hypersurface.  \\

 The type-I (capillary) Morse index, Definition \ref{typeIMIdef}, is the index of $Q(\cdot, \cdot)$ in the space of smooth functions with zero average\footnote{It is also called weak Morse index for CMC surfaces. The capillary hypersurface is weakly stable if and  only if  the type-I Morse index is $0$.}. Applying our abstract formulation in \cite{TZindexconstraintsI20}  yields the following (see Theorem \ref{app2}). 

\begin{theorem}
	\label{app1}
	Let $\Sigma\subset \Omega$ be a capillary hypersurface. Then the type-I Morse index is equal to $\text{MI}(Q)-1$ if and only if there is a smooth function $u$ such that 
	\begin{equation*}
	\begin{cases}
	(\Delta+p)u &=-1 \text{ on } \Sigma,\\
	\nabla_\eta u &=q u \text{ on } \partial \Sigma,\\
	\int_{\Sigma}u &\leq 0. 
	\end{cases}
	\end{equation*}
	Otherwise, it is equal to $\text{MI}(Q)$.
\end{theorem}
\begin{remark}
	We note that our analysis is applicable and easier for the cases of closed hypersurfaces and the fixed boundary problem. See Section \ref{closedcase} for more details. In particular, we obtain a generalization of \cite{koise01} and \cite{souam19}.  
\end{remark}
\begin{remark}
	Also, either case might happen. See Subsection \ref{subcylinder} for examples. 
\end{remark}

Next, we consider the type II partitioning problem. The setup is exactly the same as the Type-I described above except for the constraint. Here, instead of fixing the enclosed volume, we fix the wetting area. Somewhat surprisingly, the index form is given by the same bilinear form as above. Thus, the type-II Morse index, Definition \ref{typeIIMIdef}, is the index of $Q(\cdot, \cdot)$ in the space of smooth functions with zero boundary average. Applying the abstract formulation yields the following.  
\begin{theorem}
	\label{app3}
	Let $\Sigma\subset \Omega$ be a stationary hypersurface of type II partitioning. Then its type-II Morse index is equal to $\text{MI}(Q)-1$ if and only if there is a smooth function $u$ such that 
	\begin{equation*}
	\begin{cases}
	{(\Delta+p)}u &=0 \text{ on } \Sigma,\\
	\nabla_\eta u -u&= 1 \text{ on } \partial \Sigma,\\
	\int_{\Sigma}u &\leq 0. 
	\end{cases}
	\end{equation*}
	Otherwise, it is equal to $\text{MI}(Q)$.
\end{theorem}

When considering the partitioning of the Euclidean ball, $\Omega=\mathbb{B}^{n+1}$, we introduce a generalization of type I and type II constraint, called type I+II. It essentially corresponds to the partitioning of a convex body when preserving both the wetting area and enclosed volume. The type I partitioning problem is its Lagrange multiplier version. A corollary is that we streamline the type I and type II stability results of \cite{WX19} and \cite{GX20} as special cases.\\

The type-I+II Morse index of a FBMS, Definition \ref{typeIIIMIdef}, is the index of $Q(\cdot, \cdot)$ in the space of functions with zero boundary average \text{and} zero average. We will give an estimate of index by the number of nonnegative eigenvalues of a matrix calculated by the geometry of $\Sigma$.  Let $X: \Sigma^n\to \mathbb{R}^{n+1}$ be  an immersed  capillary hypersurface in the Euclidean unit ball $\mathbb{B}^{n+1}$. Let $\circA$ denote the traceless second fundamental form. 

\begin{definition}
	A capillary hypersurface is called $|\circA|^2$-scale equivalent to a hyper-planar domain if it is not umbilical and $|\circA|^2 X$ is on a hyperplane.  
\end{definition}
\begin{remark}  A capillary hypersurface is $|\circA|^2$-scale equivalent to a hyper-planar then its is on half-ball and the level sets of $|\circA|^2$ are hyper-planar. We do not know any example of such hypersurfaces.  
\end{remark}
 For any coordinate system of $\mathbb{R}^{n+1}$, $\{ e_1,\cdots ,e_{n+1}\}$, we can restate Theorem \ref{thm1} as the following.
\begin{theorem}\label{thm1.6}
	Assume $X: \Sigma\to \mathbb{R}^{n+1}$ is an immersed  capillary hypersurface in the Euclidean unit ball $\mathbb{B}^{n+1}$. Let $\ell$ be the number of nonnegative eigenvalues of the matrix  $\Upsilon$:
	$\Upsilon=(\Upsilon_{ij})_{(n+1)\times(n+1)}$ with
\[
\begin{split}
\Upsilon_{ij}= \int_\Sigma n|\circA|^2[\langle (n-H\langle X,\nu\rangle)X+(n\cos\theta+\frac{H}2(|X|^2+1))\nu,e_i\rangle\langle X,e_j\rangle] 
\end{split}
\]
where $|\circA|$ is the  norm of the traceless second fundamental form of $\Sigma$.
	Then
	\begin{enumerate}
		\item If $\Sigma$ is type-I+II stable (index zero) then it is totally umbilical.
		\item If it is $|\circA|^2$-scale equivalent to a hyper-planar domain then its type-I+II Morse index is greater than or equal to $\ell-1$.   
		\item Otherwise, the type-I+II Morse index is greater than or equal to $\ell$. 
	\end{enumerate}
\end{theorem}

\begin{remark}
	The first part recovers results of  Ros-Vergasta \cite{RV95}, Wang-Xia \cite[Theorem 1.1]{WX19}, and \cite{GX20} as special cases.
\end{remark}
\begin{remark}
	Theorem \ref{thm1.6} can be applied to estimate Morse indices for special examples. In some cases,  $\circA$, coordinate functions, and a normal vector are easy to compute. And we see in Section \ref{capillaryEuclidean} that, for round cylinders, the estimate is sharp.  
\end{remark}

We also apply our technique to study the indices of capillary minimal surfaces. When the angle is right, they are called free boundary minimal surfaces (FBMS). It is observed that a FBMS is a critical point for either type I, type II, or type I+II partitioning problem. So the following result might be of independent interest. \\

\begin{corollary}\label{weakMIfbms}
	Let $X: \Sigma\to \mathbb{R}^{n+1}$ be  an immersed  free boundary minimal hypersurface in the Euclidean unit ball $\mathbb{B}^{n+1}$. If it has type I+II Morse index less than $n+1$, then it must be totally geodesic.
\end{corollary}

An immediate consequence is the following.
\begin{corollary}\label{weakMIfbms1}
	Let $X: \Sigma\to \mathbb{R}^{n+1}$ be  an immersed  free boundary minimal hypersurface in the Euclidean unit ball $\mathbb{B}^{n+1}$. If it has type I or type II Morse index less than $n+1$, then it must be totally geodesic.
\end{corollary}

\begin{remark}
	Corollary \ref{weakMIfbms1} can also be deduced directly from the index estimate of \cite{FS16, tran16index}. Also there is a general version for capillary minimal hypersurfaces when the intersecting angle is not necessarily right (see Corollary \ref{cor7.5}).
\end{remark}

\subsection{Precise Index Computations}
It is an interesting problem to determine   the precise  Morse index for important examples. The computations have been done for critical catenoids as free boundary minimal hypersurfaces(see  see \cite{SZ16index, SSTZ17morse, tran16index, tran17gauss}). Also, in \cite{tran16index}, the first author develops a scheme to determine the Morse index of a bi-linear form associated with an elliptic boundary consideration. A slight generalization with a simpler proof is obtained in \cite{TZindexconstraintsI20} using the abstract formulation. So, together with results here, there is a procedure to precisely determine Morse indices with constraints. We give a few examples here. 
\begin{corollary}
	\label{precise1}
	Let $\Sigma\subset\mathbb{S}^{n+1}$ be a closed CMC surface of constant scalar curvature. Then, its weak Morse index is equal to 
	\[\text{MI}(Q)-1.\]
\end{corollary}
\begin{remark}
	The weak Morse index is formally given by Definition \ref{weakMIdef}.
\end{remark} 
\begin{remark}
 When $n=2$, the scalar curvature is a multiple of the intrinsic Gauss curvature. Thus, readers can consult \cite{RS07} for a precise computation of $\text{MI}(Q)$ when the Gaussian curvature is vanishing. 	
\end{remark}	
	
  For capillary cases, the simplest nontrivial examples we know are round cylinders. It happens in critical cateniods that their indices become surprisingly high when the dimensions are increasing \cite{SSTZ17morse}. But for the capillary round cylinders, even for a fixed dimension, the Morse indices can be arbitrarily large when the radii are close to $1$ or $0$. More precisely, from Proposition \ref{prop-cyl} we have
\begin{theorem}  
	\label{precise3}
	For a round cylinder $Z:=\{(x,z)\in \mathbb{R}^{n}\times \mathbb{R},\,   |x|^2=r^2, ~~r<1\}\cap \mathbb{B}^{n+1}$, 
	\begin{enumerate}
		\item $\text{MI}(Q) \geq n+2$;
		\item When $r\rightarrow 0$ or $r\rightarrow 1$, $\text{MI}(Q)\rightarrow \infty$;
		\item There is an interval $0<a<r<b<1$ such that $\text{MI}(Q)=n+2$.  
	\end{enumerate} 
\end{theorem}

The critical catenoid  can be considered as a capillary surface with zero mean curvature and right intersecting angle.  

\begin{theorem}
	\label{precise2}
	Let $\Sigma\subset \mathbb{B}^3$ is the critical catenoid. Then its type-I, type-II, and type I+II Morse indices are all equal to 3. 
\end{theorem}

The organization of the paper is as follows. The next section will collect some preliminaries and fix our notations. In Section 3, we study the set-up of closed CMC hypersurfaces, emphasizing ideas when the calculation is relatively easy and proving Corollary \ref{precise1}. The next two sections are devoted to type I and type II partitioning problems, respectively, and prove Theorems \ref{app1} and \ref{app3}. Then, we'll look into the Euclidean case in details in Section 6, where there are proofs of Theorems \ref{thm1.6} and \ref{precise3}. The last Section is about free boundary minimal hypersurfaces and obtain Corollary \ref{weakMIfbms} and Theorem \ref{precise2}. 
\section{Preliminaries}
\label{prelim}
First, we record our notations, conventions, and collect useful results.

\subsection{Variational Formulae on Hypersurfaces} 
Let $\Omega^{n+1}$ be a Riemannian manifold with or without boundaries and $X:\Sigma^n \mapsto \Omega^{n+1}$, an isometric immersion of an orientable $n$-dimensional compact manifold $\Sigma$. We also identify $\Sigma$ with its image $X(\Sigma)$. When $\Omega$ has boundaries and $\partial\Sigma=\Sigma\cap \partial \Omega$ we assume $\Omega$ is diffemorphic to an Euclidean ball. Then $\Sigma\subset \Omega$ is called a partitioning of $\Omega$. We now denote  by $\bar \n$, $\bar \Delta$ and $\bar \n^2$ (correspondingly $\n$, $\Delta$ and $\n^2$) the gradient, the Laplacian and the Hessian on $\Omega$ (on $\Sigma$) respectively. 

Let's consider a differentiable family of immersions:
\begin{align*}
X(t,\cdot): &(-\epsilon, \epsilon)\times \Sigma \rightarrow \Omega,\\
X(t, \text{int}\Sigma) &\subset \text{int} \Omega,\\
X(t, \partial \Sigma) &\subset \partial \Omega,\\
X(0,\cdot) &=X. 
\end{align*}

Next, we recall various functionals. First, the area functional is given by 
\[A(t) = \int_\Sigma d\mu(t).\]
Here  $d\mu(t)$ is  the area element of $\Sigma$ with respect to the pullback metric via $X(t,\cdot)$.
The volume functional is given by
\[V(t) =\int_{[0,t]\times \Sigma} X^{\ast} d\mu_{\Omega},\]
where   $d\mu_{\Omega}$ is  the volume element of $\Omega$.
The wetting area functional is given by
\[W(t) = \int_{[0,t]\times \partial \Sigma} X^{\ast} d\mu_{\partial \Omega},\]
 where $d\mu_{\partial \Omega}$ the volume element of $\partial\Omega$.
Also, it is physically relevant to consider the followng energy functional, for a real number $\th\in (0,\pi)$,
\[E(t) = A(t)-\cos(\th)W(t).\]

\begin{wrapfigure}{R}{0.4\textwidth}
	\begin{center}
		\includegraphics[width=0.38\textwidth]{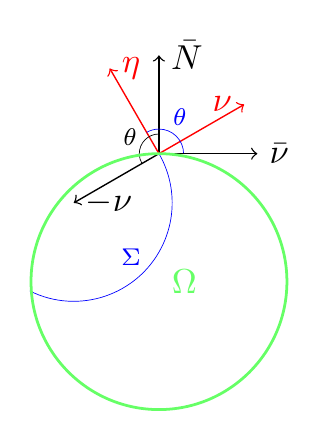}
	\end{center}
\end{wrapfigure}
For variational computation, we fix some convention. For simplicity, let's $d\mu$ and $ds$ denote the (induced) volume element of $\Sigma$ and $\partial\Sigma$ respectively. For convenience and when the context is clear, we'll omit it when writing integrals. \\

Let $\nu$ be a choice of a unit normal vector $\Sigma\subset \Omega$. $\eta$ is the exterior normal vector of $\partial\Sigma\subset \Sigma$. $\bar N$ is the outward pointing unit normal vector of $\partial\Omega\subset \Omega.$ Finally let $\bar \nu$ be the unit normal to $\partial \Sigma$ in $\partial\Omega$ such that the bases $\{\nu, \eta\}$ and $\{\bar \nu, \bar N\}$ have the same orientation in the normal bundle of $\partial \Sigma\subset \Omega$. So along $ \partial \Sigma$, the angle between $-\nu$ and $\bar N$ is equal to the one between $\eta$ and $\bar \nu$. \\

Also, the second fundamental form, for vector fields $X, Y$, is defined as, 
\[\A(X, Y)=-\left\langle{\bar{\nabla}_{X}Y, \nu}\right\rangle.\]
Then $|\A|$ denotes its norm and the mean curvature $H$ is defined to be its trace. Furthermore, let $\circA$ the traceless second fundamental form. That is, $|\circA|^2:=|\A|^2-\frac{H^2}n$.\\

The first variation of these functionals is well-known and collected below.
\begin{align*}
X'(0) &= Y,\\
A'(0) &=\int_{\Sigma}H \left\langle{\nu,Y}\right\rangle d\mu+\int_{\partial\Sigma}\left\langle{\eta,Y}\right\rangle ds,\\
V'(0) &= \int_{\Sigma}\left\langle{\nu,Y}\right\rangle d\mu,\\
W'(0) &= \int_{\partial \Sigma}\left\langle{\bar{\nu},Y}\right\rangle ds\\
E'(0) &=\int_{\Sigma}H \left\langle{\nu,Y}\right\rangle d\mu+\int_{\partial\Sigma}\left\langle{\eta-\cos(\th)\bar{\nu},Y}\right\rangle ds.
\end{align*}

For critical points of problems we consider, $\Sigma$ always intersects $\p\Omega$ at a constant angle. That is, the angle between $-\nu$ and $\bar N$ or equivalently between $\eta$ and $\bar \nu$ is everywhere equal to $\th$. Namely, 
\begin{eqnarray} 
&&\eta=\sin \th\, \bar N+\cos \th\, \bar \nu,  \label{mu0}\\&&\nu=-\cos \th\, \bar N+\sin \th\, \bar \nu.\nonumber
\end{eqnarray}
Equivalently,
\begin{eqnarray} 
&&\bar{N}=\sin \th \eta-\cos \th\, \nu,  \label{mu1}\\&&\bar\nu=\cos \th\, \eta+\sin \th\, \nu.\nonumber
\end{eqnarray}

\begin{definition}
	An orientable immersed smooth hypersurface $\Sigma^n\subset \Omega^{n+1}$ is called capillary if it has constant mean curvature and constant intersecting angle. 
\end{definition}
\begin{remark}\label{remak2} If  $\p \Omega$ is umbilical in $\Omega$ then $\eta$ is a principal direction of $\p \Sigma$ in $\Sigma$, 
	and 
	\begin{eqnarray*}
		\bar \n_\eta \nu= \A (\eta, \eta)\eta.
	\end{eqnarray*}
\end{remark}
 
Next, we collect useful calculation when $\Omega$ is Euclidean. The next lemmas are well-known and come from straightforward calculation. 
\begin{lemma} \label{lem3.1} Let $X: \Sigma\to \mathbb{R}^{n+1}$ be an isometric immersion. The following identities hold:
	\begin{eqnarray*}
		&&\Delta X=-H\nu,
		\\&&\Delta \frac12|X|^2=n-H\<X, \nu\>,
		\\&&\Delta \nu=\n H-|\A|^2\nu,
		\\&&\Delta\<X, \nu\>=\<X, \n H\>+ H-|\A|^2\<X, \nu\>,
		\\&&\<\bar\nabla X,\bar\nabla\nu\rangle= H,
	\end{eqnarray*}
\end{lemma}

\begin{corollary} \label{lem3.1} Let $X: \Sigma\to \mathbb{R}^{n+1}$ be an isometric immersion with constant mean curvature. Denote $J=\Delta+|\A|^2$. Then the following identities hold:
	\begin{eqnarray*}
		\\&&J \nu=0,\label{eq-nu}
		\\&&J\<X, \nu\>= H, \label{eq-xnu}
		\end{eqnarray*}
\end{corollary}

\begin{prop}\label{prop3.3} Let $X: \Sigma\to\mathbb{B}^{n+1}$ be a capillary hypersurface with intersecting angle $\th\in (0, \pi)$. Then along $\p \Sigma$, 
	\begin{align*}
	\bar \n_\eta( X+\cos \th\,\nu) =& q( X+\cos \th\, \nu) \label{bdry1}
	\\\bar \n_\eta Y &=q Y,
	\end{align*}
	where  
	\begin{align*}
	Y &= \<X, \nu\>X-\frac12(|X|^2+1) \nu\\
	q &=\frac{1}{\sin \th}+\cot(\theta) \A(\eta,\eta).
	\end{align*}
\end{prop}

\begin{proof} On the boundary, $X=\bar{N}$ and $X+\cos\theta\, \nu=\sin\theta\, \eta$ by (\ref{mu1}). Using Remark \ref{remak2},
	\begin{align*}
	\bar \nabla_\eta (X+\cos \th\,\nu)
	&= \eta+\cos \th\, \A(\eta, \eta)\eta=q \sin \theta\,\eta.
	\end{align*}
	
	For the second identity, we observe, on the boundary, $\bar \nabla_\eta X=\eta$, $|X|=1$ and
	\begin{align*}
	\bar \nabla_\eta \nu &=\A(\eta,\eta)\eta,\\
	\<X, \nu\>X-\frac12(|X|^2+1) \nu &=-\cos\theta\, X-\nu=-(\cos\theta\,\sin\theta\,\eta+\sin^2\theta\,\nu).
	\end{align*}  
	Then, 
	\begin{align*}
	& \bar \nabla_\eta[\<X, \nu\>X -\frac12(|X|^2+1) \nu]\\
	&=(\langle\bar \nabla_\eta X, \nu\rangle + \langle\bar \nabla_\eta \nu, X\rangle )X+\langle X,\nu\rangle\bar \nabla_\eta X-\<\bar \nabla_\eta X, X\>\nu-\frac12(|X|^2+1)\bar \nabla_\eta \nu\\
	&= \langle\bar \nabla_\eta \nu, X\rangle X+\langle X,\nu\rangle\eta -\<\eta, X\>\nu-\bar \nabla_\eta \nu\\
	&= \A(\eta,\eta)( \sin\theta\, X-\eta)-\cos\theta\,\eta -\sin\theta\,\nu
	\\
	&=-\A(\eta,\eta)\cot\theta(\cos\theta\,\sin\theta\,\eta+\sin^2\theta\,\nu)-\frac1{\sin\theta}(\cos\theta\,\sin\theta\,\eta+\sin^2\theta\,\nu)
	\\
	&= q(\cos\theta\,\sin\theta\,\eta+\sin^2\theta\,\nu)\\
	&=q(\nu+\cos\theta\ X).
	\end{align*}
	
	The result follows.
\end{proof}
 
\begin{proposition}
	\label{boundarysum}
 Let $X: \Sigma\to\mathbb{B}^{n+1}$ be a capillary hypersurface with intersecting angle $\th\in (0, \pi)$, then
	\[\int_{\p\Sigma}n \eta+ H\bar\nu =0. \]
\end{proposition}
\begin{proof}
	Let $a$ be a constant vector field.	By the divergence formula
	\begin{align*}
	\int_{\Sigma} -H\<\nu, a\> d\mu &= \int_\Sigma \Delta \<X, a\>\\
	&=\int_{\partial\Sigma} \nabla_\eta \<X, a\>\\
	&= \int_{\partial\Sigma} \<\eta, a\>. 
	\end{align*}
	On the other hand, for vector field $Y_a=\<X, \nu\> a-\<\nu, a\> X$, its divergence is 
	\[\text{div}Y_a=-n\<v, a\>. \]
	Therefore, by the divergence theorem again and (\ref{mu1}), 
	\begin{align*}
	\int_{\Sigma} -n\<\nu, a\> d\mu &= \int_\Sigma \text{div} Y_a\\
	&=\int_{\partial\Sigma} \<\eta, Y_a\>\\
	&= \int_{\partial\Sigma} -\cos\th \<\eta, a\>-\sin\theta\<\nu, a\>\\
	&= \int_{\partial\Sigma} -\<\cos\th \eta+\sin\theta\nu, a\>\\
	&= \int_{\partial\Sigma} -\<\bar\nu, a\>.  
	\end{align*}
	Combining identities above yield the desired result. 
\end{proof}
\begin{remark}
	Another way to intepret this identity is that the vector field $n\nabla \<X, a\>+H(\<\nu, a\>X-\<X,\nu\>a$ is divergence-free. 
\end{remark}

\subsection{The Bilinear Index Form and An Abstract Formulation}
When $\Sigma$ is capillary, it is somewhat surprising that the second variations for problems we consider will have the same formula. Indeed, for $u=\<Y, \nu\>$, it is show that for appropriate variations satisfying type I or type II constraint, the second variation for functional $E$ is given by, see \cite{RS97, BCJ88, GX20}, 
\begin{equation*}
E''(0) = \int_{\Sigma}|\nabla u|^2-p u^2 d\mu-\int_{\partial \Sigma}q u^2 ds, 
\end{equation*}
for
\begin{align*}
p &:= \text{Rc}^{\Omega}(\nu,\nu)+|\A|^2,\\
q &:= \frac{1}{\sin\th}\A^{\partial\Omega}(\bar{\nu},\bar{\nu})+\cot\th \A^{\Sigma}(\eta,\eta).
\end{align*}
Here, $\A^{\partial\Omega}$ and $\A^{\Sigma}$ are the second fundamental forms with respect to $\bar N$ and $\nu$ respectively. As a consequence, the index form is defined as 
\begin{equation}
\label{indexform}
Q(u,v) = \int_{\Sigma}\left\langle{\nabla u,\nabla v}\right\rangle-p uv d\mu-\int_{\partial \Sigma}q uv ds. 
\end{equation}

We recall the index definition.
\begin{definition}
	The index of a symmetric bi-linear form in a vector space is the maximal dimension of a subspace on which the form is negative definite. 
\end{definition}
\begin{definition}
	$\text{MI}(Q)$ is the index of $Q(\cdot, \cdot)$ in the space of smooth functions on $\Sigma$.  
	\end{definition}
By PDE theory, the following eigenvalue problem with Robin boundary data 
\begin{equation}
\label{Robin}
\begin{cases}
{J}u &=-\lambda u \text{  on  } \Sigma,\\
\nabla_\eta u -qu & =0 \text{  on  }\partial \Sigma.
\end{cases}
\end{equation}
has countable eigenvalues going to infinity. And $\text{MI}(Q)$ is exactly equal to the number of negative eigenvalues counting multiplicities.\\

As we see in \cite{TZindexconstraintsI20}, the index stays the same if replacing the space of smooth functions by $H^1(\Sigma)=W^{1,2}(\Sigma)$, the Sobolev space with one derivative and $L^2$-norm, which is a Hilbert space. In \cite{TZindexconstraintsI20}, we show how the index changes when restricting to a Hilbert subspace of finite codimension. For completeness, we briefly recall the results here.  

Let $H$ be a separable Hilbert space with an inner product $(\cdot, \cdot)$ and $S(\cdot, \cdot)$ is a symmetric continuous bi-linear form and $\phi_i$ is a continuous linear functional on $H$.
The inner product on $H$ induces a linear map $\mathcal{S}$ from $H$ to its continuous dual $H^\ast$ such that, for all $v\in H$, 
\[S(u,v)=(\mathcal{S}u)(v).\] 
Via the Riesz representation theorem, $H^\ast$ can be equipped with an inner product so that it is isometric to $H$. 

\begin{definition} Let $\text{ran}(\mathcal{S})$ be the range of $\mathcal{S}$, $\overline{\text{ran}(\mathcal{S})}$ its closure by the induced norm. $\overline{\text{ran}(\mathcal{S})}-\text{ran}(\mathcal{S})$ is called the set of pure limit points. 
\end{definition}

\begin{theorem}\cite[Theorem 1.1]{TZindexconstraintsI20}
	\label{abstractMorse}
	Let $H$ be a separable Hilbert space and $S(\cdot, \cdot)$ is a continuous symmetric bi-linear form. Then for any non-trivial continuous linear functional $\phi$ such that $\phi$ is not a pure limit point, we have 
	\[ \text{MI}^\phi(S)=\begin{cases}
	\text{MI}(S)-1 \text{   if there is $u\in H$ such that $\mathcal{S}u=\phi$ and $\phi(u)\leq 0$}\\
	\text{MI}(S) \text{  otherwise.}
	\end{cases}
	\]
\end{theorem}
We say $\phi$ is $S$-critical if there is $u\in H$ such that $\mathcal{S}u=\phi$ and $\phi(u)\leq 0$. It turns out that the assumption on whether $\phi$ is a pure limit point is closely related to a Fredholm alternative (see next subsection). Also, there is a version for several functionals.  

\begin{theorem}
	\label{severalfunc}\cite[Theorem 1.3]{TZindexconstraintsI20}
	Let $H$ be a separable Hilbert space and $S(\cdot, \cdot)$ is a continuous symmetric bi-linear form. Suppose that, for $i=1,...n$,
	\[\mathcal{S}(u_i)=\phi_i,\]
	and $\{\phi_i\}_{i=1}^n$ are linearly independent. Then, \[\text{MI}^{\phi_1,...,\phi_n}(S)=\text{MI}(S)-c,\]
	where $c$ is the number of non-positive eigenvalues of the symmetric matrix $S(u_i,u_j)$. In particular, 
	\[ \text{MI}(S)\geq c.\]
\end{theorem}

Theorem \ref{severalfunc} can  be applied to estimate the Morse index for  capillary hypersurfaces in a Euclidean ball. We'll rewrite it in this context. 

\begin{theorem}\label{thmindex}
	Assume $X: \Sigma\to \mathbb{R}^{n+1}$ is an immersed capillary hypersurface in the Euclidean unit ball $\mathbb{B}^{n+1}$. Let $\varphi_1,...\varphi_m$ be independent continuous linear functionals in $H^1(\Sigma)$ and $G=\cap_{i=1}^m \text{Ker}(\varphi_i)$.  
	Let $\{(u_i\}_{i=1}^{k}$ be $C^2$ functions so that
	$$ \begin{cases}
&u_i\in G;\\
&Ju_i=\phi_i, \textrm{ on } \Sigma;\\
&\nabla_\eta u_i=q u_i, \textrm{ on } \partial\Sigma.
\end{cases}  $$
Furthermore, assume that $\varphi_1,...\varphi_m, \bar\phi_1,...\bar\phi_k$ are linearly independent where $\bar\phi_i$ is the linear functional defined by $L^2(\Sigma)$-multiplication by $\phi_i$. Then \[\text{MI}^{\lbrace\varphi_1,...\varphi_m, \bar\phi_1,\cdots,\bar\phi_k\rbrace}(Q)=\text{MI}^{\lbrace\varphi_1,...\varphi_m\rbrace}(Q)-i_k,\] where  $i_k$ is the number of nonnegative eigenvalues of the matrix 
$$\Upsilon:=\left( \int_\Sigma u_i\phi_j\right).$$
	In particular, $\text{MI}^{\lbrace\varphi_1,...\varphi_m\rbrace}(Q)\geq i_k$.
	
\end{theorem}
\begin{proof}
	Since $G$ is a subspace of finite codimension, it is a Hilbert space. Because  $\varphi_1,...\varphi_m, \bar\phi_1,...\bar\phi_k$ are linearly independent,  $\bar\phi_1,...\bar\phi_k$ are linearly independent as linear functionals on $G$. Recall that
		\[ Q(u,v):= \int_{\Sigma}\nabla u\nabla v-puv-\int_{\partial \Sigma}quv.\]
		By integration by parts, 
		\begin{align*}
		Q(u_i, u_j) &=\int_{\Sigma}-Ju_i u_j+\int_{\partial\Sigma} (\nabla_\eta u_i-qu_i)u_j\\
		&= \int_\Sigma u_i\phi_j.
		\end{align*}
 	Applying Theorem \ref{severalfunc} for linear independent functionals $\bar\phi_1,...\bar\phi_k$ on $G$ finishes the proof. 
\end{proof}

Also, another application of the abstract formulation is the decomposing of the Robin boundary value problem to ones with simpler conditions. An earlier version is given by the first author in \cite{tran16index}.  First, one consider only variations fixing the boundary. The associated fixed boundary problem is given by the following Dirichlet consideration:  
\begin{equation}
\label{FBP}
\begin{cases}
{J}v &=-\delta v \text{ on } \Sigma,\\
v &= 0 \text{ on } \partial \Sigma.
\end{cases}
\end{equation}

The influence of the boundary is, then, captured by the Jacobi-Steklov problem. Suppose that $q\in C^\infty(\partial \Sigma)$ be a non-zero non-negative function. We consider:  
\begin{equation}
\label{Jacobi-Steklov}
\begin{cases}
{J}h &=0 \text{ on } \Sigma,\\
\nabla_\eta h &= \mu q h \text{ on } \partial \Sigma.
\end{cases}
\end{equation}

\begin{theorem}\cite[Theorem 4.1]{TZindexconstraintsI20}\cite[Theorem 3.3]{tran16index}
	\label{indexdecompose}
	Let $(\Sigma, \partial \Sigma)$ be a smooth compact Riemannian manifold with boundary and $q\geq 0$, $q\not\equiv 0$. Then $\text{MI}(Q)$ is equal to  
	\[a+b.\]
	Here $a$ is the number of non-positive eigenvalues of (\ref{FBP}) counting multiplicity; $b$ is the number of eigenvalues smaller than $1$ of (\ref{Jacobi-Steklov}) counting multiplicity.  
\end{theorem}

\subsection{Fredholm Alternative with Robin Boundary Condition}
\label{Fredholm} 
In this subsection, we will derive a Fredholm alternative for an elliptic operator with a Robin boundary condition. The presentation here follows \cite{evans10}. First, we recall the abstract Fredholm alternative. 

\begin{theorem}[Abstract Fredholm Alternative]\cite[Appendix D]{evans10}
	Let $K:H\mapsto H$ be a compact linear operator on a Hilbert space and $K^\ast$ its adjoint. Then
	\begin{itemize}
		\item The null space of $\text{Id}-K$ is finite dimensional,
		\item The range of $\text{Id}-K$ is closed,
		\item The range of $\text{Id}-K$ is the orthogonal complement of the null space of $\text{Id}-K^\ast$,
		\item The null space of $\text{Id}-K$ is trivial if and only if $\text{Id}-K$ is onto,
		\item The dimension of the null space of $\text{Id}-K$ is equal to that of $\text{Id}-K^\ast.$ 
	\end{itemize}
\end{theorem}

For a compact domain $\Sigma$ with a smooth boundary, let $J=\Delta+p$ for a smooth bounded function $p$. The trace operator $T:H^1(\Sigma)\mapsto L^2(\partial \Sigma)$ is such that, for $u\in C^\infty({\Sigma})$, which is dense in $H^1(\Sigma)$, 
\[ Tu=u\mid_{\partial \Sigma}.\]
The trace theorem then asserts that
\[ ||Tu||_{L^2(\partial \Sigma)}\leq c(\Sigma)||u||_{H^1(\Sigma)}.\] 
Similarly, we define $D$ to be the extension of the normal derivative $D: H^1(\Sigma)\mapsto L^2(\partial \Sigma)$. For $u\in  C^\infty({\Sigma})$ and $\eta$ the outward conormal vector along $\partial \Sigma$, 
\[Du=\nabla_\eta u.\] 
Given function $f\in L^2(\Sigma)$ and $g\in L^2(\partial\Sigma)$, one considers the system
\begin{equation}
\label{generalRobin}
\begin{cases}
{J}u &= f \text{  on  } \Sigma,\\
D_\eta u -qu & =g \text{  on  }\partial \Sigma.
\end{cases}
\end{equation}
The associated homogeneous problem is given by
\begin{equation}
\label{homRobin}
\begin{cases}
{J}u &= 0 \text{  on  } \Sigma,\\
D_\eta u -qu & =0 \text{  on  }\partial \Sigma.
\end{cases}
\end{equation}
The weak solution approach is concerned with the appropriate symmetric bilinear form on $H^1(\Sigma)$:
\[ Q(u,v):= \int_{\Sigma}\nabla u\nabla v-puv-\int_{\partial \Sigma}qT(u)T(v),\]
for all $u, v\in H^1(\Sigma)$. This bilinear form induces an abstract linear map from $H^1(\Sigma)$ to its continuous dual $(H^1)^\ast(\Sigma)$ such that
\[Q(u,v):=\mathcal{Q}u(v).\] 

For $u, v\in C^\infty({\Sigma})$, by part integration, one observes 
\[Q(u,v)= -\int_{\Sigma} vJu +\int_{\partial \Sigma}(D(u)-qT(u))T(v). \]
So, formally, one see that, 
\[\mathcal{Q}u=(-Ju, D(u)-qT(u))\] 
with the functional acting by $L^2$-inner product on the interior and on the boundary:
\begin{align*}
Q(u,v)&=\mathcal{Q}u(v)\\
&=(-Ju, D(u)-qT(u))(v)\\
&=(-Ju, v)_{L^2(\Sigma)}+(Du-gT(u), v)_{L^2(\partial\Sigma)}.
\end{align*}
Indeed, $u$ is called a weak solution of (\ref{generalRobin}) if, for all $v\in H^1(\Sigma)$,
\begin{align*}
(\mathcal{Q}u)v &=(-f,g)v
\end{align*}
as linear functionals. That is, 
\begin{align*}
Q(u,v) &=(-f,g)v= (-f,v)_{L^2(\Sigma)}+(g, Tv)_{L^2(\partial\Sigma)}.
\end{align*}

Towards a weak solution, the Lax-Milgram theorem guarantees it via the Riesz representation theorem if the bi-linear form is bounded and coercive (positive definiteness). It is clear that $Q(\cdot, \cdot)$ is bounded (by Cauchy-Schwarz inequalities and the trace theorem above) but not necessarily coercive. However, since $|p|$ is bounded and $H^1(\Sigma)\subset L^2(\Sigma)$, one can modify 
\[ Q_\gamma(u,v)=Q(u,v)+\gamma\<u,v\>_{L^2(\Sigma)}\]
such that $S_\gamma$ is coercive for some positive constant $\gamma$. That is, there is a positive constant $\beta$ such that
\[Q_\gamma(u,u)\geq \beta||u||^2_{H^1(\Sigma)}. \]
Essentially, $Q_\gamma(\cdot,\cdot)$ corresponds to the operator $\mathcal{Q}_\gamma=(J_\gamma, D(u)-qT(u))$ for $J_\gamma=\gamma-J$. Since $Q_\gamma(\cdot,\cdot)$ is bounded and coercive, $\mathcal{Q}_\gamma$ is an isomorphism from $H^1(\Sigma)$ to $(H^1)^\ast(\Sigma)$. 

We have, if $u$ is a solution of (\ref{generalRobin}) then 
\begin{align*}
\mathcal{Q}_\gamma u&= (-Ju +\gamma u, D(u)-qT(u)) \\
&=(-f+\gamma u, g).
\end{align*}
Thus,
\begin{align*}
u &=\gamma \mathcal{Q}_\gamma^{-1} (u, 0)+ \mathcal{Q}_\gamma^{-1}(-f, g)\\
&=Ku+f_1,\\
K &:=\gamma \mathcal{Q}_\gamma^{-1}(\cdot, 0)\\
f_1 &:= \mathcal{Q}_\gamma^{-1}(-f, g).
\end{align*}
Here, by the natural identification $L^2(\Sigma)\subset (H^1)^\ast(\Sigma)$ and $H^1(\Sigma)\subset L^2(\Sigma)$, 
\[K: L^2(\Sigma) \mapsto L^2(\Sigma).\]
Then, one observes that, for $Ku=v$
\begin{align*}
\mathcal{Q}_\gamma(v)&=\gamma (u,0),\\
\beta ||v||_{H^1(\Sigma)}^2 &\leq Q_\gamma(v,v)=\mathcal{Q}_\gamma v(v)\\
&=\gamma (u,0)(v) =\gamma(u,v)_{L^2(\Sigma)}\\
&\gamma \leq ||u||_{L^2(\Sigma)}||v||_{L^2(\Sigma)}\\
&\leq ||u||_{L^2(\Sigma)}||v||_{H^1(\Sigma)}.
\end{align*}
Thus,
\[ ||v||_{H^1(\Sigma)}=||Ku||_{H^1(\Sigma)}\leq C||u||_{L^2(\Sigma)}. \]
Due to the Rellich-Kondrachov compactness theorem, $K$ is a compact operator. So the abstract Fredholm alternative theorem is applicable and  can be translated to the following.

\begin{theorem}
	\label{FredholmRobin}
	For any $f\in L^2(\Sigma)$ and $g\in L^2(\partial\Sigma)$, either (\ref{generalRobin}) has a unique weak solution or the homogeneous problem (\ref{homRobin}) has a non-trivial space $N$ of weak solutions. Furthermore, if the latter holds, then $N$ has finite dimension and (\ref{generalRobin}) has a weak solution if and only if $(f,v)_{L^2(\Sigma)}=(g,v)_{L^2(\partial\Sigma)}$ for all $v\in N$.
\end{theorem}
\begin{proof}
	The only nontrivial part is to interpret $f_1$ in the orthogonal complement of the null space of $\text{Id}-K^\ast$. For $u, w\in L^2(\Sigma)$ let 
	\begin{align*}
	Ku &=\gamma \mathcal{Q}^{-1}_\gamma(u, 0) = x,\\
	\mathcal{Q}_\gamma (x) &=\gamma (u,0),\\
	Kw &=\gamma \mathcal{Q}^{-1}_\gamma(w, 0) = y,\\
	\mathcal{Q}_\gamma (y) &=\gamma (w,0).
	\end{align*}
	We have
	\begin{align*}
	(Ku, w)_{L^2(\Sigma)} &= (x, w)_{L^2(\Sigma)}\\
	&=(w,0)(x)=\frac{1}{\gamma}\mathcal{Q}_\gamma (y)(x)\\
	&=\frac{1}{\gamma}{Q}_\gamma (y,x)=\frac{1}{\gamma}\mathcal{Q}_\gamma x(y)\\
	&= (u,0)(y) =(u,y)_{L^2(\Sigma)}=(u,Kw)_{L^2(\Sigma)}.
	\end{align*}
	Thus, $K^\ast=K$. Then for $v$ in the null space of $\text{Id}-K^\ast=\text{Id}-K$ we have
	\begin{align*}
	v &=\gamma \mathcal{Q}_\gamma^{-1}(v,0),\\
	\mathcal{Q}_\gamma(v) &= \gamma (v,0),\\
	0 &= (f_1, v)_{L^2(\Sigma)}= (v,0)(f_1)\\
	&=	\mathcal{Q}_\gamma(v)(f_1) =S_\gamma(v, f_1)\\
	&=\mathcal{Q}f_1(v)=(-f,g)v\\
	&=-(f,v)_{L^2(\Sigma)}+(g,v)_{L^2(\partial\Sigma)}.
	\end{align*} 
	The result then follows. 
\end{proof}

\begin{remark} When there is no boundaries, Theorem \ref{FredholmRobin} just recovers the Fredholm alternative of an elliptic operator on a compact closed manifold. When $q=0$, it recovers the Fredholm alternative of an elliptic operator with a Neumann boundary condition.  
\end{remark}

\section{Morse index and Weak Morse index}
\label{closedcase}
 In this section, we consider a the area functional for two-sided hypersurfaces of codimension one with the enclosed volume constraint in an ambient manifold. Let $\Sigma\subset \Omega$ be a critical point. By Section \ref{prelim}, the constraint is realized as, for a test function $u$,
 \[\int_\Sigma u d\mu=0.\]
 Consequently, $\Sigma$ is characterized by having CMC. Due to the absence of boundaries,
the second variation becomes \cite{BCJ88}[Proposition 2.5]:
\begin{align*}
Q(u,v) &= \int_{\Sigma}\Big(\left\langle{\nabla u,\nabla v}\right\rangle-p uv\Big) d\mu,\\
p &:= \text{Rc}^{\Omega}(\nu,\nu)+|\A^\Sigma|^2. 
\end{align*}  
\begin{definition}
	\label{weakMIdef}
	The weak Morse index of the hypersurface $\Sigma^n\subset \Omega^{n+1}$ is the index of $Q(\cdot, \cdot)$ on $\mathfrak{F}=\{u\in H^1(\Sigma): \int_{\Sigma}u d\mu=0\}$. 
	\end{definition} 
Using integration by parts, one observe that
\begin{align*}
Q(u,v) &= \int_{\Sigma}\left\langle{\nabla u,\nabla v}\right\rangle-p uv d\mu\\
&= -\int_{\Sigma}v(\Delta+p)u d\mu\\
&=  -\int_{\Sigma}vJu d\mu.
\end{align*}
Here,
\[J:= \Delta+p\]
is the Jacobi operator. By PDE theory, the following eigenvalue problem 
\begin{equation*}
{J}u =-\lambda u \text{  on  } \Sigma.
\end{equation*}
has countable eigenvalues going to infinity. And $\text{MI}(Q)$ is exactly equal to the number of negative eigenvalues counting multiplicities. The weak Morse index is at most equal to $\text{MI}(Q)$ and is finite. We are ready to characterize that relation.

\begin{theorem}
	\label{app5}
	Let $\Sigma\subset \Omega$ be a closed, orientable, CMC hypersurface. Its weak Morse index is equal to $\text{MI}(Q)-1$ if and only if there is a smooth function $u$ such that 
	\begin{equation}
	\begin{cases}
	(\Delta+p)u &=-1 \text{ on } \Sigma,\\
	\int_{\Sigma}u &\leq 0. 
	\end{cases}
	\end{equation}
	Otherwise, it is equal to $\text{MI}(Q)$. 
\end{theorem}
\begin{proof}
In this case, the weak index is just $\text{MI}^\phi(Q)$ with $\phi$ is the $L^2$-multiplication by the function $1$ (constantly equal to one). Let $\bar\phi$ be its correspondence via the Riesz representation theorem.  

\textbf{Claim:} $\phi$ is not a pure limit point. 

It is essentially equivalent to the Fredholm alterative associated with the elliptic operator $J=\Delta+p$; see the Remark following Theorem \ref{FredholmRobin}. Indeed, in case there is a weak solution to the equation $Ju=1$, then  $\phi\in \text{ran}(\mathcal{Q})$. If not, by the Fredholm alternative, the null space $\text{Ker}(\mathcal{Q})$ has a positive finite dimension and for some $v\in \text{Ker}(\mathcal{Q})$   
\[0\neq (1,v)_{L^2(\Sigma)}= \phi(v)=(\bar{\phi},v)_{H^1(\Sigma)}.\]
Thus, $\bar{\phi}$ is not $H^1$-perpendicular to $\text{Ker}(\mathcal{Q})$. Therefore, $\phi$ is not a pure limit point. 

The assumption of Theorem \ref{abstractMorse} is verified so we have 
\[ \text{MI}^\phi(Q)=\begin{cases}
\text{MI}(Q)-1 \text{   if $\phi$ is $Q$-critical}\\
\text{MI}(Q) \text{  otherwise.}
\end{cases}
\]
It remains to determine the $Q$-critical condition. In our case, $\mathcal{Q}(u)=(-Ju)$ as in Section \ref{Fredholm}. Thus, $Q$-critical is equivalent to the existence of a weak solution
\begin{equation*}
\begin{cases}
{-J}u &=1 \text{ on } \Sigma,\\
(1,u)_{L^2(\Sigma)} &\leq 0. 
\end{cases}
\end{equation*}
The regularity theory for an elliptic operator then asserts that the solution is smooth.  	
\end{proof}
   
In some cases, the Q-criticality can be determined from the geometry of the surface.  For instance, we will prove the following which says that the difference of the Morse index and weak Morse index is one for CMC surfaces with constant Gaussian curvature in a unit sphere.
\begin{corollary}
	\label{precise1dup}
	Let $\Sigma$ be a immersed CMC  hypersurface in $\mathbb{S}^{n+1}$  of constant scalar curvature. Then, its weak Morse index is equal to 
	\[\text{MI}(Q)-1.\]
\end{corollary}
\begin{proof}
	For a CMC hypersurface in $\mathbb{S}^{n+1}$,
	\[J=\Delta+n+|\A|^2.\]
	For $S$ denoting the scalar curvature, we recall the following Gauss equation \cite[Chapter 1]{chowluni}
	\[S_\Sigma =S_{\mathbb{S}^3}-2\text{Rc}_{\mathbb{S}^3}(\nu, \nu)+H^2-|\A|^2.\]
	Therefore 
	\[|\A|^2= H^2-S_\Sigma+n(n-1).\]
	Since $S_\Sigma$ is constant, 
	\[p=n+|\A|^2=n^2+H^2-S=\text{constant}.\]
	
	We have, $\text{Ker}(\mathcal{Q})$ is the set of solutions 
	\[(\Delta+p)u=0.\]
	It is non-trivial if and only if $p>0$ is an eigenvalue of $\Delta$. In that case, since the constant function $1$ is the first eigenfunction of $\Delta$ with eigenvalue zero, for any $v\in \text{Ker}(\mathcal{Q})$  
	\[(1,v)_{L^2(\Sigma)}=0.\]
	By the Fredhom alternative, Theorem \ref{FredholmRobin}, there is always a solution of \[Ju=-1.\] 
	For such $u$, we have,
	\[\int_\Sigma pu d\mu=p\int_\Sigma u d\mu=\int_{\Sigma}(-1-\Delta u)d\mu=-\text{Area}(\Sigma)<0.\]
	The assertion then follows from Theorem \ref{app5}.
\end{proof}
\begin{remark} For $\Sigma^2\subset \mathbb{S}^{3}$, we can replace the scalar curvature by the intrinsic Gauss curvature $K$
\[|\A|^2= H^2-2(K-1)=2+H^2-2K.\]
\end{remark}

An almost identical analysis replacing $H^1(\Sigma)$ by $H^1_0(\Sigma)$ is applicable for the fixed boundary problem. The statement goes as follows. 
\begin{theorem}
	\label{fbpindex}
	Let $\Sigma\subset \Omega$ be a  CMC hypersurface with boundaries. Then its weak Morse index with respect to the fixed boundary problem is equal to $\text{MI}(Q)-1$ if and only if there is a smooth function $u$ such that
	\begin{equation}
	\begin{cases}
	{J}u &=-1 \text{ on } \Sigma,\\
	u &= 0 \text{ on } \partial \Sigma,\\
	\int_{\Sigma}u &\leq 0. 
	\end{cases}
	\end{equation}
	Otherwise, it is equal to $\text{MI}(Q)$.  
\end{theorem}
\begin{remark}
	This is a generalization of results from \cite{koise01} and \cite{souam19}.
\end{remark}

\section{Type-I Partitioning}
\label{sec2}

In this section, we study the Morse index associated with the type I-partitioning problem. By the terminology of Section \ref{prelim}, $X:\Sigma^n \mapsto \Omega^{n+1}$, is a type $I$ partitioning if it is the critical point of the energy functional $E$ while preserving prescribed volume $V$.

\begin{definition}
	A variation $Y$ is volume-preserving if 
	\[\int_{\Sigma}\left\langle{\nu,Y}\right\rangle d\mu=0.\]
\end{definition}

It is clear from the first variation computation that $\Sigma$ is a critical point if and only if it has constant mean curvature $H$ and the projection of $\eta$ onto the tangent bundle of $\partial\Omega$ is exactly equal to $(\cos{\th})\bar{\nu}$. That is, $H$ is constant and $\Sigma$ intersects $\partial \Omega$ at constant angle $\th$. So critical points are capillary hypersurfaces. It is also noted that for each smooth function $u$ on $\Sigma$ with $\int_\Sigma ud\mu=0$, there exists an admissible volume-preserving variation $Y$ of $\Sigma$ with $u=\left\langle{Y, \nu}\right\rangle$ \cite{RS97}.
\begin{definition}
	\label{typeIMIdef}
	The Type-I (capillary) Morse index of the hypersurface $\Sigma^n\subset \Omega^{n+1}$ is the index of $Q(\cdot, \cdot)$ on $\mathfrak{F}:=\left\{u\in H^1(\Sigma);  \int_\Sigma u d\mu=0\right\}$.  
\end{definition}

A capillary hypersurface is called type-I stable if its type I index is zero. That is, 
 \begin{eqnarray*}
	Q(u, u)\ge 0,\qquad \forall u\in \mathfrak{F}:=\left\{u\in H^1(\Sigma);  \int_\Sigma u d\mu=0\right\}.
\end{eqnarray*}

 The type I (capillary) Morse index is at most equal to $\text{MI}(Q)$ and is finite. We are ready to exactly determine their difference.

\begin{theorem}
	\label{app2dup}
	Let $\Sigma\subset \Omega$ be a capillary hypersurface. Then the type-I Morse index is equal to $\text{MI}(Q)-1$ if and only if there is a smooth function $u$ such that 
	\begin{equation*}
	\begin{cases}
	(\Delta+p)u &=-1 \text{ on } \Sigma,\\
	\nabla_\eta u &=q u \text{ on } \partial \Sigma,\\
	\int_{\Sigma}u &\leq 0. 
	\end{cases}
	\end{equation*}
	Otherwise, it is equal to $\text{MI}(Q)$.
\end{theorem}
\begin{proof} [Proof of Theorem \ref{app2}]
	
The capillary Morse index is equal to $\text{MI}^\phi(Q)$ on $H_1$ where $\phi$ is the functional 
	\[\phi(u)=(1, u)_{L^2(\Sigma)}.\] 
	Let $\bar\phi$ be its Riesz representation. 
	
	\textbf{Claim:} $\phi$ is not a pure limit point. 
	
	We apply Theorem \ref{FredholmRobin} for $f=1$, $g=0$. In case there is a weak solution, then  $\phi\in \text{ran}(\mathcal{Q})$. If not, the null space $\text{Ker}(\mathcal{Q})$ is non-trivial and has finite dimension and for some $v\in \text{Ker}(\mathcal{Q})$   
	\[(\bar{\phi},v)_{H^1(\Sigma)}=\phi(v)=(f,v)_{L^2(\Sigma)}\neq 0.\]
	Thus, $\bar{\phi}$ is not $H^1$-perpendicular to $\text{Ker}(\mathcal{Q})$. Therefore, $\phi$ is not a pure limit point. 
	
	The assumption of Theorem \ref{abstractMorse} is verified so we have 
	\[ \text{MI}^\phi(Q)=\begin{cases}
	\text{MI}(Q)-1 \text{   if $\phi$ is $Q$-critical}\\
	\text{MI}(Q) \text{  otherwise.}
	\end{cases}
	\]
	It remains to determine the $Q$-critical condition. In our case, $\mathcal{Q}(u)=(-Ju, Du-qu)$ as in Section \ref{Fredholm}. Thus, $Q$-critical is equivalent to the existence of a weak solution
	\begin{equation*}
	\begin{cases}
	{-J}u &=1 \text{ on } \Sigma,\\
	D u &= qu \text{ on } \partial \Sigma,\\
	(1,u)_{L^2(\Sigma)} &\leq 0. 
	\end{cases}
	\end{equation*}
	The Schauder's regularity theory for elliptic boundary-value problems then asserts that the solution is smooth.  
\end{proof}
Furthemore, combining Theorem \ref{app1} and Theorem \ref{indexdecompose} yields the following.
\begin{corollary}
	\label{main1}
	Let $\Sigma\subset \Omega$ be a capillary surface and $q\geq 0$, $q\not\equiv 0$. Then its type I Morse index is equal to 
	\[ a+b+c.\]
	Here, $a$ is the number of non-positive eigenvalues of (\ref{FBP}) counting multiplicity; $b$ is the number of eigenvalues smaller than $1$ of (\ref{Jacobi-Steklov}) counting multiplicity. Finally, $c=-1$ if there is a smooth function $u$ such that 
	\begin{equation}
	\begin{cases}
	{J}u &=-1 \text{ on } \Sigma,\\
	\nabla_\eta u &= qu \text{ on } \partial \Sigma,\\
	\int_{\Sigma}u &\leq 0. 
	\end{cases}
	\end{equation}
	Otherwise, $c=0$. 
\end{corollary}
\begin{proof}
	It follows from Theorem \ref{app2dup} and Theorem \ref{indexdecompose}.	 
\end{proof}
\section{Type-II Partitioning}
In this section, we investigate Type II Morse indices for stationary hypersurfaces. The setup of type II partitioning is similar to type I. The difference is that, instead of preserving prescribed volume, type II preserves prescribed wetting area. Following \cite{GX20}, we have the following definitions. 
\begin{definition}
	A variation $Y$ is wetting-area-preserving if 
	\[\int_{\partial\Sigma}\left\langle{\bar{\nu},Y}\right\rangle d\mu=0.\]
\end{definition}
\begin{definition}
	$\Sigma$ is a stationary point of type II partitioning if $E'(0)=0$ for all wetting-area-preserving variations. 
\end{definition}
From our variational formulae in Section \ref{prelim}, it is clear that a type-II stationary surface is minimal and meets the boundary at a constant angle $\th$. For $Y=Y_0+u \nu$, with $Y_0$ tangential, it is observed that the wetting-area-preserving is equivalent to
\[\int_{\partial\Sigma}u ds=0.\]

\begin{definition}
	\label{typeIIMIdef}
	The Type-II Morse index of the hypersurface $\Sigma^n\subset \Omega^{n+1}$ is the index of $Q(\cdot, \cdot)$ on $\mathfrak{G}=\{u\in H^1(\Sigma): \int_{\partial \Sigma}u d\mu=0\}$.    
\end{definition}

Since the index form is the same as type I, the type II Morse index is at most $\text{MI}(Q)$ and is finite. We are ready to determine it precisely. 

\begin{theorem}
	\label{app3dup}
	Let $\Sigma\subset \Omega$ be a stationary hypersurface of type II partitioning. Then its type-II Morse index is equal to $\text{MI}(Q)-1$ if and only if there is a smooth function $u$ such that 
	\begin{equation*}
	\begin{cases}
	{(\Delta+p)}u &=0 \text{ on } \Sigma,\\
	\nabla_\eta u -qu&= 1 \text{ on } \partial \Sigma,\\
	\int_{\Sigma}u &\leq 0. 
	\end{cases}
	\end{equation*}
	Otherwise, it is equal to $\text{MI}(Q)$.
\end{theorem}
\begin{proof}
The type II Morse index is equal to $\text{MI}^\varphi(Q)$ on $H_1$ where $\varphi$ is the functional 
	\[\varphi(u)=(1, u)_{L^2(\partial\Sigma)}.\] 
	Let $\bar{\varphi}\in H^1(\Sigma)$ be its Riesz representation. 
		
	\textbf{Claim:} $\varphi$ is not a pure limit point. 
	
	We apply Theorem \ref{FredholmRobin} for $f=0$, $g=1$. In case there is a weak solution then, clearly, $\varphi\in \text{ran}(\mathcal{Q})$. If not, the null space $\text{Ker}(\mathcal{Q})$ is non-trivial and has finite dimension and for some $v\in \text{Ker}(\mathcal{Q})$   
	\[(\bar{\varphi},v)_{H^1(\Sigma)}=\varphi(v)=(g,v)_{L^2(\partial\Sigma)}\neq 0.\]
	Thus, $\bar{\varphi}$ is not $H^1$-perpendicular to $\text{Ker}(\mathcal{Q})$. Therefore, $\varphi$ is not a pure limit point. 
	
	The assumption of Theorem \ref{abstractMorse} is verified so we have 
	\[ \text{MI}^\varphi(Q)=\begin{cases}
	\text{MI}(Q)-1 \text{   if $\varphi$ is $Q$-critical}\\
	\text{MI}(Q) \text{  otherwise.}
	\end{cases}
	\]
	
	It remains to determine the $Q$-critical condition. In our case, $\mathcal{Q}(u)=(-Ju, Du-qu)$ as in Section \ref{Fredholm}. Thus, $Q$-critical is equivalent to the existence of a weak solution
	\begin{equation*}
	\begin{cases}
	{-J}u &=0 \text{ on } \Sigma,\\
	D u -qu&= 1 \text{ on } \partial \Sigma,\\
	(1,u)_{L^2(\partial\Sigma)} &\leq 0. 
	\end{cases}
	\end{equation*}
	The Schauder's regularity theory for elliptic boundary-value problems then asserts that the solution is smooth. 
\end{proof}
\begin{corollary}
	\label{indextypeII}
	Let $\Sigma\subset \Omega$ be a critical point of type II partitioning. Then its type II Morse index is equal to 
	\[ a+b+c.\]
	Here, $a$ is the number of non-positive eigenvalues of (\ref{FBP}) counting multiplicity; $b$ is the number of eigenvalues smaller than $1$ of (\ref{Jacobi-Steklov}) counting multiplicity. Finally, $c=-1$ if there is a smooth function $u$ such that 
	\begin{equation}
	\begin{cases}
	{J}u &=0 \text{ on } \Sigma,\\
	\nabla_\eta u -u&= 1 \text{ on } \partial \Sigma,\\
	\int_{\Sigma}u &\leq 0. 
	\end{cases}
	\end{equation}
	Otherwise, $c=0$. 
\end{corollary}

\begin{proof}
	It follows from Theorem \ref{app3dup} and Theorem \ref{indexdecompose}.	 
\end{proof}

\section{Capillary hypersurfaces in a Euclidean ball}
\label{capillaryEuclidean}
In this section, we study capillary hypersurfaces in a Euclidean ball. We introduce the following generalization of type I and type II constraints, called type I+II. It essentially corresponds to the partitioning of a convex body while preserving both the wetting area and enclosed volume. Minimizers of such problem will be studied somewhere else. Here, we focus on the index estimate. A corollary is that we streamline the type I and type II stability results of \cite{WX19} and \cite{GX20} as special cases.  
\begin{definition}
	A function $u$ satisfies type I+II constraint if simultaneously
	\begin{align*}
	0 &= \int_{\partial\Sigma} u ds,\\
	0 &= \int_{\Sigma} u d\mu.
	\end{align*}
\end{definition}

\begin{definition}
	\label{typeIIIMIdef}
	The Type-I+II Morse index of the capillary hypersurface $\Sigma^n\subset \Omega^{n+1}$ is the index of $Q(\cdot, \cdot)$ on the vector space $\{u\in H^1(\Sigma): \int_{\partial \Sigma}u d s=0,~~ \int_{\Sigma}u d\mu=0\}$.     
\end{definition}
\begin{remark}
	Obviously, for a given capillary hypersurface, the type I or type II Morse index is greater than or equal to the type I+II index. 
\end{remark}
In addition, we consider the case that $\Omega^{n+1}=B^{n+1}\subset \mathbb{R}^{n+1}$ is the unit ball. The goal is to construct a family of functions satisfying type I+II constraint. We would like to highlight that our computation is valid and simpler if the boundary $\partial \Sigma$ is empty. So our results can be applied to closed  CMC hypersurfaces in Euclidean spaces.

\begin{prop} \label{prop3.1} Let $X: \Sigma\to \mathbb{R}^{n+1}$ be an isometric immersion with constant mean curvature. Let, for $\phi=\frac{|X|^2}2+c$, $c_1=n\cos\theta$,  
	\begin{align*} 
	w: &=\langle X,\nu\rangle,\\
	Z: &=(\Delta \phi)X-\phi\Delta X,\\
	\Phi: &=H\phi-nw-c_1-2cH.
	\end{align*}
Then, for $c_2$ any constant, the following identities hold:
	\begin{eqnarray}
	&& J (Z+c_2\nu)=n|\circA|^2 X\label{eq-Xnu},\\
	&& \Delta \Phi=n|\circA|^2w,\\
	&& \langle (Z+c_1\nu),J(Z+c_1\nu)\rangle=n^2|\circA|^2|x^T|^2-\Phi\Delta\Phi.\label{eq-Z1}\\
	&&  \bar{\nabla}_\eta( Z+c_1\nu)|_{\partial \Sigma}=q(Z+c_1\nu)|_{\partial \Sigma},
	\end{eqnarray}
\end{prop}
\begin{proof} 
	We first prove \eqref{eq-Xnu}. First, $Z=(\Delta \phi)X-\phi\Delta X=X\Delta\phi+\phi H\nu$. Since $\phi=\frac{|x|^2}2+c$, $\Delta \phi=n-Hw$. Therefore,
	\[
	\begin{split}
	\Delta Z&=X(-H\Delta w)+H\phi\Delta \nu\\
	&=-H X(H-|\A|^2w)+H\phi(-|\A|^2\langle \nu,a\rangle)\\
	&=-H^2 X+H|\A|^2w X+\phi |\A|^2\Delta X\\
	&=(n|\A|^2-H^2)X-|\A |^2[(n-Hw)X-\phi\Delta X]
	\end{split}
	\]
	Then 
	\[
	\Delta Z+|\A|^2Z=n|\circA|^2X.
	\]

   Next, 
	\[
	\Delta \Phi=H(n-Hw)+n(|\A|^2w-H)=n|\circA|^2w.
	\]
	Then,
		\[\begin{split}
	\<(Z+c_1\nu),X\>&=|X|^2\Delta\phi+\phi Hw+c_1w\\
	&=n|X|^2-H|x|^2w+\phi Hw+c_1w\\
	&=n|X^T|^2+nw^2-H(2\phi-2c)w+\phi Hw+c_1w\\
	&=n|X^T|^2-w\Phi. 
	\end{split}
	\]
	Therefore,
	\[\begin{split}
	\langle (Z+c_1\nu),J(Z+c_1\nu)\rangle&=n^2|\circA|^2|X^T|^2-n\Phi|\circA|^2 w\\
	&=n^2|\circA|^2|x^T|^2-\Phi\Delta\Phi
	\end{split}
	\]
	The last identity follows from Proposition \ref{prop3.3}.
\end{proof}

We now fix a coordinate system of $\mathbb{R}^{n+1}$,  $\phi(X)=\frac12(|X|^2+1)$, and choose, for $i=1,2,\cdots, n+1$,  
\begin{align}
u_i &=(Z+c_1\nu)_i,\\
Ju_i &= \phi_i=n|\circA|^2 X_i.
\end{align}
\begin{lemma}
	\label{satisfyingconstraints}
	Each $u_i$ satisfies the type I+II constraint. 
\end{lemma}
\begin{proof}
	Indeed, we have
	\begin{align*}
	Z_i &=(\Delta \phi)X_i-\phi\Delta X_i=\ddiv(X_i\nabla\phi -\phi\nabla X_i),\\
	n\nu_i &= \ddiv (\nu_i X-w e_i).
	\end{align*}
	Therefore, by the divergence theorem, 
	\begin{align*}
	\int_\Sigma u_i &= \int_{\Sigma} \ddiv (X_i\nabla\phi -\phi\nabla X_i+\cos\th (\nu_i X-w e_i))d\mu\\
	&= \int_{\partial\Sigma} (X_i \<X,\eta\>-\eta_i+\cos\th(\nu_i\<X, \eta\>-w \eta_i)) ds
	\end{align*}
	Applying (\ref{mu0}) and (\ref{mu1}) yields
	\begin{align*}
	\int_\Sigma u_i &= \int_{\partial\Sigma} (-\cos\th \nu_i+\cos\th \nu_i)ds\\&=0. 
	\end{align*}
	
	Next, we consider $u_i$ on the boundary $\partial\Sigma$. Using the calculation above and (\ref{mu1}) we have 
	\begin{align*}
	u_i\mid_{\partial\Sigma} &= (n-H\<X,\nu\>)X_i+\frac{1}{2}(X^2+1)H\nu_i+n\cos\th \nu_i\\
	&=(n+H\cos\th)X_i+H\nu_i+ n\cos\th\nu_i\\
	&=n \<X+\cos\th \nu, e_i\>+H\<\cos\th X+\nu, e_i\>\\
	&= n\sin\th \<\eta, e_i\>+H\sin\th \<\bar\nu, e_i\>.
	\end{align*}
	Thus, by Proposition \ref{boundarysum}, 
	\begin{align*}
	\int_{\p\Sigma} u_i &= 0. 
	\end{align*}
	
\end{proof}
 We consider the folloiwng matrix
 $$\Upsilon:=\left( \int_\Sigma u_i\phi_j\right)_{(n+1)\times(n+1)}.$$ 
\begin{lemma}	
	\label{traceofmatrix}	
The trace $\textrm{tr}(\Upsilon)$ satisfies:
	\[
	\textrm{tr}(\Upsilon)=\int_\Sigma \left(n^2|\circA|^2|X^T|^2+|\nabla\Phi|^2\right)\, dA.
	\]
\end{lemma}
\begin{proof}
	
	We use (\ref{eq-Z1}) in Proposition \ref{prop3.1}:	
	\[
	\sum_{i=1}^{n+1}u_i\phi_i=\langle (Z+c_1\nu),J(Z+c_1\nu)\rangle=n(n|\A|^2-H^2)|X^T|^2-\Phi\Delta\Phi.
	\]
	Since $, \phi(X)=\frac12(|X|^2+1)$, on the boundary, 
	\begin{align*}
	\Phi &=H\phi -nw-n\cos\th-H\\
	&=\frac12H(|X|^2+1)-n\langle X, \nu\rangle-n\cos\th-H\\
	&=0. 
	\end{align*}
	Then the results follows from $\int_\Sigma \Phi\Delta\Phi=-\int_\Sigma |\nabla\Phi|^2.$
\end{proof}

\begin{definition}
	A capillary hypersurface is called $|\circA|^2$-scale equivalent to a hyper-planar domain if $|\circA|^2 X$ is on a hyperplane.  
\end{definition}
\begin{remark}  A capillary hypersurface is $|\circA|^2$-scale equivalent to a hyper-planar then its is on half-ball and the level sets of $|\A|^2$ are hyper-planar.  
	\end{remark}
 
\begin{theorem}\label{thm1}
	Assume $X: M\to \mathbb{R}^{n+1}$ is an immersed  capillary hypersurface in the Euclidean unit ball $\mathbb{B}^{n+1}$. Let $\ell$ be the number of nonnegative eigenvalues of the matrix  $\Upsilon$.
	Then
	\begin{enumerate}
		\item If $X$ has zero type-I+II Morse index then it is totally umbilical.
		\item If it is $|\circA|^2$-scale equivalent to a hyper-planar domain then its type-I+II Morse index is greater than or equal to $\ell-1$.   
		\item Otherwise, the type-I+II Morse index is greater than or equal to $\ell$. 
	\end{enumerate}
\end{theorem}
\begin{proof}[Proof of Theorem \ref{thm1}]
	Let $H^1(\Sigma)\ \supset H_1:=\{u\in H, \int_\Sigma u=\int_{\p\Sigma}u=0\}$. By Lemma \ref{satisfyingconstraints}, $u_i\in H_1$. Thus, if it is type-I+II stable then $\Upsilon$ is nonpositive definite, and $\textrm{tr}(\Upsilon_1)\le 0$.	From Lemma \ref{traceofmatrix}
	\begin{eqnarray}\label{stab-eq3}
	&&\int_M \left(n(n|\A|^2-H^2)|x^T|^2+|\nabla\Phi|^2\right)\, dA\le 0.
	\end{eqnarray}
	
	Thus, $|X^T|^2(n|\A|^2-H^2)=0$ and $\n \Phi=0$ which implies that $\Phi\equiv 0$. Therefore,	$\< X, \nu\> (n|\A|^2-H^2)=\Delta\Phi\equiv 0$ and, consequently, $\circA\equiv 0$ and $\Sigma$ is totally umbilical.\\ 
	
	Second, if $\Sigma$ is not umbilical then $(n|\A|^2-H^2)\ne 0$ at some point, then there exists an opens subset $U$ on which $(n|\A|^2-H^2)\ne 0$. Let $\phi_i'$ be the linear functional defined by $L^2(\Sigma)$ multiplication of the function $\phi_i$.\\ 
	
	\textbf{Claim:} $\phi_1, \cdots,\phi_{n+1}$ are linearly independent and, thus, the dimension of the space $\text{span}\{\phi_1, \cdots,\phi_{n+1}\}$ is $n+1$. 
	
	Otherwise there exists $c_1, \cdots, c_{n+1}\in \mathbb{R}$ such that $c^2_1+ \cdots+c^2_{n+1}\ne 0$ and, by Proposition \ref{prop3.1},  
	\[
	c_1 X_1+ \cdots c_{n+1}X_{n+1}=0.
	\]
	This implies that $U$ is contained in a hyperplane and thus is totally geodesic which contradicts to $(n|\A|^2-H^2)\ne 0$. Thus, the claim is true,
	
	\textbf{Claim:} The space $\text{span}(\bar\phi_1, \cdots,\bar\phi_{n+1})$, as linear functionals on $H_1$ has dimension $n$ if $\Sigma$ is $|\circA|^2$-scale equivalent to a hyper-planar. Otherwise it has dimension $n+1$.
	
	Let $\chi$ and $\psi$ be linear functionals on $H^1(\Sigma)$ defined by $L^2(\Sigma)$ and $L^2(\p\Sigma)$-multiplication by the constant function $1$. It is immediate that 
$\{\bar\phi_1, \cdots,\bar\phi_{n+1}, \psi\}$ are linearly independent. Then, $\{\bar\phi_1, \cdots,\bar\phi_{n+1}, \chi\}$ are linearly dependent if and only if, by Proposition \ref{prop3.1}, 
\[c_1 |\circA|^2 X_1+...c_{n+1}|\circA|^2X_{n+1}+1=0.\]
That is, $|\circA|^2 \<X,a\>=1$ for some constant vector $a$. In that case, $|\circA|^2 X$ is hyper-planar. Using the previous claim, the proof is finished. \\

Finally, applying Theorem \ref{thmindex} yields the desire results.

\end{proof}

We can consider $\Upsilon$ as a bilinear form on $\mathbb{R}^{n+1}$. Namely
\begin{definition}
	The symmetric  bilinear form $\Upsilon$ on $R^{n+1}$ is defined by
	\[
	\begin{split}
	\Upsilon(v_1,v_2):&= \int_\Sigma(n|\A|^2-H^2)\langle Z+n\cos\theta\nu,v_1\rangle\langle X,v_2\rangle\\
	&= \int_\Sigma n|\circA|^2[\langle (n-H\langle X,\nu\rangle)X+(n\cos\theta+\frac{H}2(|X|^2+1))\nu,v_1\rangle\langle X,v_2\rangle] 
	\end{split}
	\]
	where $v_1, v_2\in \mathbb{R}^{n+1}$.
\end{definition}  

From the properties of the stability operator $J$, the bilineanr  form $\Upsilon$  is symmetric. Therefore we can choose an orthonormal base such that $\Upsilon$ can be diagonalized as $\textrm{diag}(\lambda_1,  \cdots,\lambda_{n+1})$
with 
\begin{equation}\label{eq-lam}
\lambda_i=\int_\Sigma n|\circA|^2(n-H\langle X,\nu\rangle)X_i^2+\int_\Sigma n|\circA|^2(n\cos\theta+\phi H)X_i\nu_i.
\end{equation}
When $H=0$, 
\begin{equation}
\begin{split}
\lambda_i&=\int_\Sigma (n|\A|^2)(n X_i^2+n\cos\theta X_i\nu_i).
\end{split}
\end{equation}
When $\cos\theta=0$, 
\begin{equation}\label{eq-lam1}
\lambda_i=\int_\Sigma n^2|\circA|^2[(n-H\langle X,\nu\rangle)X_i^2+\phi H X_i\nu_i].
\end{equation}
Therefore we can restate Lemma \ref{traceofmatrix} as the following. 
\begin{corollary}
	Let $X: \Sigma\to \mathbb{R}^{n+1}$ be  an immersed  capillary hypersurface in the Euclidean unit ball $\mathbb{B}^{n+1}$.
	Then 
	\[\sum_{i=1}^{n+1}\lambda_i=\int_\Sigma n^2\left(|\circA|^2|X^T|^2+|\circA(X^T)|^2\right).
	\]
	
\end{corollary}
\begin{proof}
	From Lemma \ref{traceofmatrix}, we have
	\[\sum_{i=1}^{n+1}\lambda_i=\int_\Sigma \left(n(n|\A|^2-H^2)|x^T|^2+|\nabla\Phi|^2\right).
	\]
	Here $\Phi=H\phi -nw-n\cos\th-H$. Then
	\[
	\begin{split}
	\nabla \Phi&=H\nabla \frac{|X|^2}{2}-n\nabla \langle X, \nu\rangle\\
	&=H X^T-n\A(X^T).
	\end{split}
	\]
	The result follows.
\end{proof}

\subsection{Cylinder in a ball}
\label{subcylinder}
We consider a flat cylinder of radius $0<r<1$ inside $\Omega=\mathbb{B}^{n+1}$, $X:\Sigma=[-T, T]\times \mathbb{S}^{n-1} \mapsto \Omega=\mathbb{B}^{n+1}$ 
\begin{equation}
\label{cylinder}
X(t, z) =(t, rz).
\end{equation}
Here $T=\sqrt{1-r^2}$. Let $\{w_i, i=1,...n-1\}$ be a basis of tangent vectors on $\mathbb{S}^{n-1}$. It is straightforward to compute tangent vectors of the cylinder:
\begin{align*}
X_{,t}&= (1, \vec{0}),\\
X_{,i}&= (0, r w_i). 
\end{align*}
Consequently,  
\begin{align*}
\nu&=(0, z),\\
\eta &=\frac{t}{|t|}X_t.
\end{align*} 
Thus, the boundary derivative becomes
\[\nabla_\eta=\partial_{\pm t}.\]
By our convention, 
\begin{align*}
\sin\th &=\<X. \eta\>=\sqrt{1-r^2}=T,\\
\cos\th &=-\<X, \nu\>=-r.
\end{align*}
Thus, \[q=\frac{1}{\sin\th}+\cot\th \A(\eta, \eta)=\frac{1}{\sqrt{1-r^2}}.\]
The second fundamental form is
\begin{align*}
\A(X_t, X_t) &= -\<X_{tt}, \nu\>=0,\\
\A(X_i, X_i) &= -\<X_{ii}, \nu\>=r,\\
\end{align*}
Thus, 
\begin{align*}
H &= (n-1)\frac{r}{|X_i|^2}=\frac{n-1}{r},\\
|\A|^2 &=\frac{n-1}{r^2},\\
|\circA|^2 &=|\A|^2-\frac{H^2}{n}=\frac{n-1}{nr^2}.
\end{align*}
The Jacobi operator is
\[J=\Delta+|\A|^2=\partial_t^2+\frac{1}{r^2}\Delta_{\mathbb{S}^{n-1}}+\frac{n-1}{r^2}.\]
We will first consider the index without any constraint, $\text{MI}(Q)$. Recall that it is the number of negative eigenvalues of the Robin boundary problem: 
\begin{align*}
J u +\lambda u &=0 \text{ on $\Sigma$},\\
\nabla_\eta u &=qu \text{ on $\partial\Sigma$}.
\end{align*}
By separation of variables, 
\begin{align*}
u &= f(t)g(z)
\end{align*}
It is then followed that $g$ is a spherical Laplacian eigenfunction 
\begin{align*}
0 &=\Delta_{\mathbb{S}^{n-1}}g +\alpha g,\\
0 &= f''+\beta f,\\
\beta &:=\lambda+\frac{n-1}{r^2}-\frac{\alpha}{r^2},\\
\frac{f(T)}{f'(T)} &=-\frac{f(-T)}{f'(-T)}=T.
\end{align*}
Therefore, we exhaust all possible cases: 
\[\begin{cases}
	\beta=0,\\
	\beta>0,~~ \tan(\sqrt{\beta}T)=\sqrt{\beta}T \text{ or} \cot(\sqrt{\beta}T)=-\sqrt{\beta}T,\\
	\beta<0,~~ \coth(\sqrt{-\beta}T)=\sqrt{-\beta}T. 
\end{cases}\]
Furthermore, by spherical harmonic analysis, $\alpha=k(k+n-2)$ with multiplicity 
\begin{align*}
m_0 &=1,\\
m_1 &=n,\\
m_k &= \frac{(n+k-1)!}{(n-1)!k!}- \frac{(n+k-3)!}{(n-1)!(k-2)!}. 
\end{align*}

\begin{proposition}\label{prop-cyl}   For a cylinder, 
	\begin{enumerate}
		\item $\text{MI}(Q) \geq n+2$;
		\item When $r\rightarrow 0$ or $r\rightarrow 1$ $\text{MI}(Q)\rightarrow \infty$;
		\item There is an interval $0<a<r<b<1$ such that $\text{MI}(Q)=n+2$.  
	\end{enumerate} 
\end{proposition}
\begin{proof}
	Recall that $\text{MI}(Q)$ is the number of negative eigenvalues $\lambda$ for 
	\[\lambda=\beta+\frac{1}{r^2}(k-1)(k+n-1),\]
	for $k=0, 1, 2...$.  Next, we exhaust all possible cases.
	
	\textbf{Case 1.} If $\beta=0$ then choosing $k=0$ yields $\lambda<0$. If $k\geq 1$ then $\lambda\geq 0$.
	
	\textbf{Case 2.} If $\beta<0$, choosing $k=0, 1$ yields $\lambda<0$. Then, let $T_0$ be the unique positive number such that 
	\[ \coth(T_0)=T_0.\]
	Then, for $k\geq 2$, 
	\[\lambda=\frac{1}{r^2}(k-1)(k+n-1)-\frac{T_0^2}{T^2}.\]
    $\lambda<0$ for all $k$ such that 
	\[(k-1)(k+n-1)<\frac{r^2 T_0^2}{1-r^2}.\]
	Therefore, $r\rightarrow 1$, $\text{MI}(Q)\rightarrow \infty$.  
		 
	 \textbf{Case 3.} If $\beta>0$ we first note that there are infinitely many periodic values satisfying the equation $\tan(\sqrt{\beta}T)=\sqrt{\beta}T \text{ or} \cot(\sqrt{\beta}T)=-\sqrt{\beta}T$. Next one observes that, for $k\geq 1$, $\lambda>0$. For $k=0$, 
	 \[\lambda=\beta-\frac{n-1}{r^2}.\]
	 Therefore,  $r\rightarrow 0$, $\text{MI}(Q)\rightarrow \infty$. Let $T_1$ be the smallest real positive number such that $\cot(T_1)=-T_1$.  $\text{MI}(Q)=n+2$ if and only if
	\begin{align*}
	n+1 &>\frac{r^2 T_0^2}{1-r^2},\\
	n-1 &<\frac{r^2 T_1^2}{1-r^2}.
	\end{align*}
	Equivalently, 
	\[\frac{1}{1+\frac{T_1^2}{n-1}}<r^2 <\frac{1}{1+\frac{T_0^2}{n+1}}. \]
\end{proof}
\begin{remark}
	Numerically, $T_0 \approx1.19968$ and $T_1\approx 2.79838$. 
\end{remark}
\begin{lemma}
	\label{nosolution}
	Let $x=T\sqrt{n-1}/r$, if $\cos(x)+x\sin{x}=0$ then the system 
	\begin{align*}
	J u &=-1,\\
	\nabla_\eta u &=qu
	\end{align*}
	has no solution. 
\end{lemma}
\begin{remark}
	It is noted that as $r$ varies from $0$ to $1$, \[x=T\sqrt{n-1}/r=\sqrt{n-1}\frac{\sqrt{1-r^2}}{r}\] goes from $+\infty$ to $0$. The function $\cos(x)+x\sin{x}$ fluctuates and assumes all possible real values when $x>0$.
\end{remark}

\begin{proof}
	By our separation of variables earlier, $\lambda=0$ (when $\beta=0$ and $k=1$) is an eigenvalue of 
	\begin{align*}
	J u +\lambda u&=0,\\
	\nabla_\eta u &=qu.
	\end{align*}
	Under the hypothesis, let $\sqrt{\beta}=\sqrt{n-1}/r$, then $\cot(\sqrt{\beta}T)=-\sqrt{\beta}T$ and $u=\cos(\sqrt{\beta}t)$ is an eigenfunction with eigenvalue $0$ of the homogeneous system above. It is observed that 
	\[\int_\Sigma u d\mu \neq 0.\]
	So, by the Fredholm alternative Theorem \ref{FredholmRobin}, the result follows. 
\end{proof}
\begin{proposition}
	\label{weakindexcyl}
	Let $x=T\sqrt{n-1}/r$, the type-I Morse index of a cylinder is
\[	\begin{cases}
		\text{MI}(Q) \text{ if either } \cos(x)+x\sin{x}=0 \text{ or }  x<\frac{\sin{x}}{x\sin{x}+\cos{x}}, \\
		\text{MI}(Q)-1 \text{ otherwise.}
	\end{cases} 
\]
\end{proposition}
\begin{proof}
	If $\cos(x)+x\sin{x}=0$ then the result follows from Lemma \ref{nosolution} and Theorem \ref{app1}. 
	
	Otherwise, $\cos(x)+x\sin{x}\neq 0$ and let 
	\[u= c\cos(t\sqrt{n-1}/r)-\frac{r^2}{n-1}.\]
	It is readily verified that $u$ solves the system 
	\begin{align*}
	J u &=-1,\\
	\nabla_\eta u &=qu
	\end{align*} 
	for $c(x\sin{x}+\cos{x})=\frac{r^2}{n-1}$. As $u_{tt}+\frac{n-1}{r^2}u=-1$,
	\[\frac{n-1}{r^2}\int_\Sigma u d\mu=-2u'(T)-2T.\] 
	We have
	\begin{align*}
	\sqrt{n-1}/r(u'(T)+T) &=-c\frac{n-1}{r^2}\sin{x}+x\\
	&= -\frac{\sin{x}}{x\sin{x}+\cos{x}}+x. 
	\end{align*}  
	Applying Theorem \ref{app1} again yields the conclusion. 
\end{proof}

Next we check how Theorem \ref{thm1.6} is applicable in this case. For these cylinders,
\begin{align*}
\<X, \nu\> &= r.
\end{align*}
Thus components of the matrix from Theorem \ref{thm1.6} become
\begin{align*}
\Upsilon_{ij} &= \int_{\Sigma} n\frac{n-1}{nr^2}\<(n-\frac{n-1}{r}r)X+ (n(-r)+\frac{n-1}{2r}(t^2+r^2+1))\nu, e_i\>\<X, e_j\>\\
&=\frac{n-1}{r^2}\int_{\Sigma} \<X+ \frac{(n-1)(t^2+1)-(n+1)r^2}{2r}\nu, e_i\>X_j\\
&=\frac{n-1}{r^2}\int_{\Sigma}X_iX_j +  \frac{(n-1)(t^2+1)-(n+1)r^2}{2r}\nu_i X_j. 
\end{align*}
For $X_1=t, X_i=rz_{i-1}$, $i>1$, $\nu_1=0$, $\nu_i=z_{i-1}$ and spherical symmetry, it is clear that for $i\neq j$
\[\Upsilon_{ij}=0.\]
It remains to calculate the diagonal terms. First, we have
\begin{align*}
\Upsilon_{11} &=\frac{n-1}{r^2}\int_{\Sigma} t^2 >0.
\end{align*}
Then, for $j=i+1>1$,
\begin{align*}
\Upsilon_{jj} &=\frac{n-1}{r^2}\int_{\Sigma} r^2 z_i^2+ \frac{(n-1)(t^2+1)-(n+1)r^2}{2r}rz_i^2  \\
&=\frac{(n-1)^2}{2r^2}\int_{\Sigma}(t^2+1-r^2) r^2 z_i^2 >0. 
\end{align*}
Thus, Theorem \ref{thm1.6} implies the type I+II index is at least $n+1$. 

\begin{proposition} Let $a<r<b$ as in part (3) of Proposition \ref{prop-cyl}. Let $x=T\sqrt{n-1}/r$ and assume that   $\cos(x)+x\sin{x}\neq 0$ and $x>\frac{\sin{x}}{x\sin{x}+\cos{x}}$. Then, the type I+II index is equal to the type I index and equal to $n+1$. 
\end{proposition}
\begin{remark}
	So in this case, the estimate from Theorem \ref{thm1.6} is sharp. 
\end{remark}
\begin{proof}
	Let $m_{I+II}$ and  $m_{I}$ be the indices with respect to constrains I+II and I respectively. First, it is clear that 
	\[m_{I+II}\leq m_I.\]
	By the calculation above and Theorem \ref{thm1.6}, $m_{I+II}$ is at least $n+1$. By Propositions \ref{prop-cyl} and \ref{weakindexcyl}, $m_I$ is at most $n+1$. The result then follows. 
\end{proof}
\begin{remark}
It is clearly possible to do an analogous analysis between type I+II and type II indices. We leave it for the reader. 	
\end{remark}

\section{Capillary Minimal Hypersurfaces}
In this section, we study the case $H=0$. Here, it is expected that the appropriate calculation earlier will be simplified significantly and we can obtain more precise results. First, we have the following. 
  
\begin{corollary}\label{cor7.5}
	Assume that  $X: \Sigma\to \mathbb{R}^{n+1}$ is a properly immersed  capillary minimal hypersurface in the Euclidean unit ball $\mathbb{B}^{n+1}$. If it is not $|\circA|^2$-scale to a hyper-planar domain and
	\[
	\int_\Sigma |\A|^2 \langle X,v\rangle^2\geq \int_\Sigma |\A|^2\cos^2\theta, 
	\]
	for any unit vector $v\in \mathbb{R}^{n+1}$,   then either  it is totally geodesic or it has type-I+II Morse index bigger than  $n+1$. 
\end{corollary}
\begin{proof}Assume that $\Sigma$ is not totally geodesic. Then there exists a coordinate $X_1, X_2, \cdots, X_{n+1} $, such that $|\A|^2 X_1, |\A|^2 X_2, \cdots, |\A|^2 X_{n+1}$ are linearly independent. 
	From (\ref{eq-lam}) we know 
	
	\begin{equation*}
	\begin{split}
	\lambda_i&=\int_\Sigma (n|\A|^2)(n X_i^2+n\cos\theta X_i\nu_i).
	\end{split}
	\end{equation*}
	By Cauchy inequality and $|\nu_i|\le 1$, we have
	\[
	|\int_\Sigma |\A|^2\cos\theta X_i\nu_i|\le \int_\Sigma |\A|^2|\cos\theta| |X_i|\le \left( \int_\Sigma |\A|^2\cos^2\theta\right)^{\frac12} \left( \int_\Sigma |\A|^2 X_i^2\right)^{\frac12}.
	\]
	Hence 
	\[
	\begin{split}
	\lambda_i&=n^2\left[\int_\Sigma |\A|^2 X_i^2+\int_\Sigma |\A|^2\cos\theta X_i\nu_i\right]
	\\
	&\ge n^2\left[\int_\Sigma |\A|^2 X_i^2-\left( \int_\Sigma |\A|^2\cos^2\theta\right)^{\frac12} \left( \int_\Sigma |\A|^2 X_i^2\right)^{\frac12}\right]\\
	&= n^2\left(\int_\Sigma |\A|^2 X_i^2\right)^{\frac12}\left[\left(\int_\Sigma |\A |^2\langle X,e_i\rangle^2\right)^{\frac12}-\left( \int_\Sigma |\A|^2\cos^2\theta\right)^{\frac12}  \right]\\
	&\geq0.
	\end{split}
	\]
	
	From Theorem \ref{thm1}, we know that the type-I+II Morse index is at least $n+1$. The proof is complete.
\end{proof}
	
\begin{corollary}
	Assume that  $X: \Sigma\to \mathbb{R}^{n+1}$ is an embedded capillary minimal hypersurface in the Euclidean unit ball $\mathbb{B}^{n+1}$. If it is not $|\circA|^2$-scale to a hyper-planar domain and is a polar-graph then either  it is totally geodesic or it has type-I+II Morse index bigger than  $n+1$. 
\end{corollary}
\begin{proof}
	We can choose $\nu$ such that the support function $\<X,\nu\>$ is positive at some point. We note that, by (\ref{mu0}), along $\p \Sigma$,
	\[|\<X,\nu\>|=|\cos\th|. \]
	Furthermore, by Corollary \ref{lem3.1},
	\[ (\Delta+|\A|^2)\<X,\nu\>=0.\] 
	Therefore, if $\<X,\nu\>$ attains a local minimum in the interior of $\Sigma$, it must be non-positive. On the other hand, since $\Sigma$ is a polar-graph, its support function $\<X,\nu\>$ must be non-negative and positive in the interior of $\Sigma$. As a consequence, $\<X,\nu\>$ attains its global mimimum on the boundaries. Therefore, everywhere on $\Sigma$,
		\[|\<X,\nu\>|\geq |\cos\th|. \]  
	The hypothesis of Corollary \ref{cor7.5} is verified and the result follows. 
\end{proof}

When $H=0$ and $\th=\pi/2$, $\Sigma$ is called a free boundary minimal surface. As $\cos\th=0$, the functional $E$ no longer depends on the wetting area and the boundaries are so-called free. 
 
\begin{corollary}\label{weakMIfbmsdup}
	Let $X: \Sigma\to \mathbb{R}^{n+1}$ be  an immersed  free boundary minimal hypersurface in the Euclidean unit ball $\mathbb{B}^{n+1}$. If it has type I+II Morse index less than $n+1$, then it must be totally geodesic.
\end{corollary}
\begin{proof}
	Assume for sake of contradiction that the immersion is not totally geodesic. First, we claim that $\Sigma$ is not $|\circA|^2$-scale to a hyper-planar domain. If it is, then for some constant unit vector $e_i$ and some fixed number $c>0$ 
	\[|\circA|^2\<X, e_i\>=c.\]
	That is $X_i$ is positive. On the other hand, by Lemma \ref{satisfyingconstraints}, 
	\[\int_\Sigma X_i=0,\]
	which is a contradiction.

	Then from the proof of Theorem \ref{thm1}, we know that $\phi_1, \cdots,\phi_{n+1}$ are linearly independent and dimension of span$\{\phi_1, \cdots,\phi_{n+1}\}$ is $n+1$ as linear functionals acting on $H_1$. 
	Since $\Sigma$ a minimal free boundary hypersurface, then $H=0$, $\cos\theta=0$. (\ref{eq-lam}) implies
	\[
	\lambda_i=\int_\Sigma n^2|\A|^2X_i^2\geq 0.
	\]
	If the type-I+II Morse index is less than $n+1$,  there exists at least one $i\in\{1,2,..., n+1\}$ such that $X_i\equiv 0$ on some open subset $U$. Hence $|\A|\equiv 0$ on $U$. This is a contradiction to the unique continuation of minimal surface and assumption that $\Sigma$ is not totally geodesic. The proof is complete. 
\end{proof}

\subsection{Critical Catenoid in a ball}
In this section, we'll determine precisely the indices with different constraints for the critical catenoid, the unique (up to isometry) rotationally symmetric free boundary minimal surface in $\mathbb{B}^3$. 
Here $\th=\frac{\pi}{2}$ and $\Omega=\mathbb{B}^3$. Then,

\begin{align*}
p &= |\A|^2,\\
q &= 1. 
\end{align*} 

The surface can be parametrized by a conformal harmonic map $X:\Sigma \mapsto \mathbb{B}^3$, for $t\in [-T, T]$ and $0\leq \tau\leq 2\pi$, 
\begin{equation}
\label{cricat}
X(t,\tau) =c(\cosh{t}\cos{\tau}, \cosh{t}\sin{\tau}, t).
\end{equation}
$T$ and $c$ are determined by
\begin{align*}
\cosh{T}&=T\sinh{T},\\
c &=\frac{1}{T\cosh{T}}.
\end{align*}
Following a straightforward calculation, we have 
\[|\A|^2=\frac{2}{c^2 \cosh^4{t}}.\]
The normal derivative along the boundary is given by
\[\nabla_\eta =\frac{1}{c\cosh{T}}\partial_{\pm t}=T \partial_{\pm t}.
\]  
The Jacobi operator is given by 
\[J=\Delta+p=\Delta+|\A|^2=\frac{1}{c^2\cosh^2(t)}(\partial_t^2+\partial_\tau^2+\frac{2}{\cosh^2(t)}).\]

Similarly, for type II index we prove the following. 
\begin{proposition}
	\label{type1cat}
	The type-I Morse index of the critical catenoid is 3.
\end{proposition}
\begin{proof}
	We consider 
	\[u= -a\cosh^2{t}+ b(1-t\tanh(t)).\]
	One then calculate
	\begin{align*}
	J u &=\frac{-4a}{c^2},\\
	D_\eta u (T)&=T(-a\sinh(2T)+b(-\tanh(T)-\frac{T}{\cosh^2(T)})) .
	\end{align*}
	Solving $D_\eta u(T)= u(T)$ gives
	\[b=-a\cosh(T)\sinh(T).\]
	Therefore, choosing $a=\frac{c^2}{4}$ and $b=-a\cosh(T)\sinh(T)$ yields
	\begin{align*}
	Ju &=-1 \text{ on } \Sigma\\
	\nabla_\eta u &=u \text{ on } \partial\Sigma,\\
	\int_\Sigma u &
	<0.  
	\end{align*}
	By Theorem \ref{app1}, its type-I Morse index is equal to $\text{MI}(Q)-1$. By recent results of \cite{tran16index, SZ16index, devyver16index}, $\text{MI}(Q)=4$ and the proof is complete.
\end{proof}
Now, we complete the second part of Corollary \ref{precise2}. 
\begin{proposition}
	\label{typ2cat}
	The type II Morse index of the critical catenoid is 3.
\end{proposition}
\begin{proof}
	We consider
	\[u=a(1-t\tanh(t)).\]
	One then calculate
	\begin{align*}
	J u &=0,\\
	D_\eta u &=-Ta(\tanh(T)+\frac{T}{\cosh^2(T)}).
	\end{align*}
	Solving $D_\eta u- u=1$ gives
	\[a=-\coth^2(T).\]
	For that choice of $a$
	\begin{align*}
	Ju &=0 \text{ on } \Sigma\\
	D_\eta u -u&=1 \text{ on } \partial\Sigma,\\
	\int_\Sigma u &
	<0.  
	\end{align*}
	By Theorem \ref{app3}, its type-II Morse index is equal to $\text{MI}(Q)-1$. By recent results of \cite{tran16index, SZ16index, devyver16index}, $\text{MI}(Q)=4$ and the proof is complete.
\end{proof}
\begin{theorem}
	\label{precise2dup}
	Let $\Sigma\subset \mathbb{B}^3$ is the critical catenoid. Then its type-I, type-II, and type I+II Morse indices are all equal to 3. 
\end{theorem}
\begin{proof}
	By Corollary \ref{weakMIfbmsdup}, the type I+II index is at least $3$. Since it is smaller than or equal to either the type I or type II one, by Propositions \ref{type1cat} and \ref{typ2cat}, the result follows.  
\end{proof}
\bibliographystyle{plain}
\bibliography{bioMorse}

\end{document}